\documentclass[11pt]{aims}
\usepackage{amsmath, amssymb, mathrsfs}
  \usepackage{paralist}
  \usepackage{graphics} 
  \usepackage{epsfig} 
\usepackage{graphicx}  \usepackage{epstopdf}
 \usepackage[colorlinks=true]{hyperref}
 \hypersetup{urlcolor=blue, citecolor=red}
 \usepackage[toc,page]{appendix}
\usepackage{chngcntr}

\usepackage{titlesec}
\titleformat{\section}{\Large\bfseries}{\thesection.}{4pt}{}
\titleformat{\subsection}{\large\bfseries}{\thesection.\arabic{subsection}.}{4pt}{}
\titleformat{\subsubsection}{\bfseries}{\thesection.\arabic{subsection}.\arabic{subsubsection}.}{4pt}{}
\titleformat*{\paragraph}{\bfseries}
\titleformat*{\subparagraph}{\bfseries}
\usepackage[margin=1.2in]{geometry}

  \def\currentyear{}
   \def\currentmonth{}
    \def\ppages{}
     \def\DOI{}

\newtheorem{theorem}{Theorem}[section]

\newtheorem{lemma}[theorem]{Lemma}
\newtheorem{proposition}[theorem]{Proposition}
\theoremstyle{definition}
\newtheorem{definition}[theorem]{Definition}
\newtheorem{remark}[theorem]{Remark}

\newcommand{\RN}{\mathbb{R}^N}
\newcommand{\Rb}{\mathbb{R}}

\newcommand{\Vc}{\mathcal{V}}

\newcommand{\Oc}{\mathcal{O}}

\newcommand{\Cc}{\mathcal{C}}

\newcommand{\Dc}{\mathcal{D}}

\newcommand{\Ls}{\mathscr{L}}

\newcommand{\Am}{\mathbf{A}_m}

\numberwithin{equation}{section}

\title[Blowup for a higher order semilinear parabolic equation] 
      {Construction of type I blowup solutions for a higher order semilinear parabolic equation}

\author[T. Ghoul, V. T. Nguyen, H. Zaag]{}

 \email[T. Ghoul]{teg6@nyu.edu}
 \email[V. T. Nguyen]{Tien.Nguyen@nyu.edu}
 \email[H. Zaag]{Hatem.Zaag@univ-paris13.fr}

\thanks{
\today}
\begin{document}
\maketitle

\centerline{\scshape Tej-Eddine Ghoul$^\dagger$, Van Tien Nguyen$^\dagger$ and Hatem Zaag$^\ast$}
\medskip
{\footnotesize
 \centerline{$^\dagger$New York University in Abu Dhabi, P.O. Box 129188, Abu Dhabi, United Arab Emirates.}
  \centerline{$^\ast$Universit\'e Paris 13, Sorbonne Paris Cit\'e, LAGA, CNRS (UMR 7539), F-93430, Villetaneuse, France.}
}

\bigskip

\begin{abstract} We consider the higher-order semilinear parabolic equation
$$
\partial_t u = -(-\Delta)^{m} u + u|u|^{p-1},
$$
in the whole space $\RN$, where $p > 1$ and $m \geq 1$ is an odd integer.  We exhibit type I non self-similar blowup solutions for this equation and obtain a sharp description of its asymptotic behavior. The method of construction relies on the spectral analysis of a non self-adjoint linearized operator in an appropriate scaled variables setting. In view of known spectral and sectorial properties of the linearized operator obtained by Galaktionov \cite{Grsl01}, we revisit the technique developed by Merle-Zaag \cite{MZdm97} for the classical case $m = 1$, which consists in two steps: the reduction of the problem to a finite dimensional one, then solving the finite dimensional problem by a classical topological argument based on the index theory. Our analysis provides a rigorous justification of a formal result in \cite{Grsl01}. 
\end{abstract}

\section{Introduction.}
We are interested in the semilinear parabolic equation 
\begin{equation}\label{Pb}
\left\{ \begin{array}{rl}
\partial_t u &= \Am u + u|u|^{p-1},\\
u(0)&= u_0
\end{array} \right. \quad \Am \equiv -(-\Delta)^m,
\end{equation}
where $u(t): \Rb^N \to \Rb$ with $N \geq 1$, $\Delta$ stands for the standard Laplace operator in $\Rb^N$, and the exponents $p$ and $m$ are fixed,  
$$p > 1 \quad \textup{and} \quad m \in \mathbb{N},\; m \geq 1 \;\; \textup{odd}.$$
The higher-order semilinear parabolic equation \eqref{Pb} is a natural generation of the classical semilinear heat equation $(m=1)$. It arises in many physical applications such as theory of thin film, lubrication, convection-explosion, phase translation, or applications to structural mechanics (see the Petetier-Troy book \cite{PTbook01} and references therein). 
 
By standard results the local Cauchy problem for equation \eqref{Pb} can be solved in $L^1 \cap L^\infty$ thanks to the integral representation
\begin{equation}
u(t) = \mathcal{K}_m(t) \ast u_0 + \int_0^t \mathcal{K}_m(t-s) \ast u(s)|u(s)|^{p-1}ds,
\end{equation}
where $\mathcal{K}_m(t)$ is the fundamental solution of the linear parabolic $\partial_t \mathcal{K}_m = \Am \mathcal{K}_m$, defined via the inverse Fourier transform
$$\mathcal{K}_m(x,t) = \frac{1}{(2\pi)^N}\int_{\Rb^N}e^{-|\xi|^{2m}t - \imath (\xi \cdot x)}d\xi, \quad \mathcal{K}_m(x,0) = \delta(x).$$
From Fujita \cite{FUJsut66} ($m = 1$) and Galaktionov-Pohozaev \cite{GPiumj02} ($m > 1$), we know that 
$$ p_F \triangleq 1 + \frac{2m}{N},$$
is the \textit{critical Fujita exponent} for the problem in the following sense. If $p > p_F$,  for any sufficient small initial data $u_0 \in L^1(\Rb^N) \cap L^\infty(\Rb^N)$ the Cauchy problem \eqref{Pb} admits a global solution satisfying $u(t) \to 0$ as $t \to +\infty$ uniformly in $\Rb^N$. If $p \in (1, p_F]$ and the initial data $u_0 \not \equiv 0, \int_{\Rb^N}u_0 dx \geq 0$ for $m > 1$ or $u_0 \geq 0$ for $m = 1$, then the corresponding solution to problem \eqref{Pb} blows up in some finite time $T > 0$, namely that
$$\lim_{t \to T}\|u(t)\|_{L^\infty(\Rb^N)} = +\infty.$$
Here $T$ is called the blowup time, and a point $a \in \Rb^N$ is called a blowup point if and only if there exists a sequence $(a_n, t_n) \to (a, T)$ such that $|u(a_n, t_n)| \to +\infty$ as $n \to +\infty$.  A solution of \eqref{Pb} is called Type I blowup if it satisfies
\begin{equation}\label{def:type1}
c(T-t)^{-\frac{1}{p-1}} \leq \|u(t)\|_{L^\infty} \leq C(T-t)^{-\frac{1}{p-1}},
\end{equation}
otherwise, it is of Type II blowup. In addition, we call a blowup solution \textit{self-similar} if it is of the form 
\begin{equation}\label{def:selfsol}
u(x,t) = (T-t)^{-\frac{1}{p-1}}\Phi(y), \quad y = \frac{x}{(T-t)^\frac 1{2m}},
\end{equation}
where $\Phi$ is not identically constant. Obviously, the \textit{self-similar} blowup solution is of Type I.

When $m = 1$, problem \eqref{Pb} reduces to the classical semilinear heat equation
\begin{equation}\label{eq:she}
\partial_t u = \Delta u + u|u|^{p-1},
\end{equation}
which has been extensively studied in the last four decades, and no rewiew can be exhaustive. Given our interest in the construction of solutions with a prescribed blowup behavior, we only mention previous work in this direction. The first conctructive result was given by Bricmont-Kupiainen \cite{BKnon94} who showed the existence of type I blowup solution to equation \eqref{eq:she} according to the asymptotic dynamic
\begin{equation}\label{eq:prostab}
\sup_{x \in \Rb^N} \left|(T-t)^{\frac{1}{p-1}} u(x,t) - \kappa\left(1 + \frac{(p-1)}{4p} \frac{|x|^2}{(T-t)|\log (T-t)|}\right)^{-\frac{1}{p-1}}\right| \to 0 \quad \textup{as} \;\; t \to T,
\end{equation}
for some universal positive constant $\kappa = \kappa(p)$. Note that the authors of \cite{BKnon94} also exhibited finite time blowup solutions that verify other asymptotic behaviors which are expected to be unstable. Note also that Bressan \cite{Breiumj90, Brejde92} made a similar construction in the case of an exponential nonlinearity. Later, Merle-Zaag \cite{MZdm97}  suggested a modification of the argument of \cite{BKnon94} and obtained the stability of the constructed solution verifying \eqref{eq:prostab} under small perturbations of initial data. The stability of the asymptotic behavior \eqref{eq:prostab} had been observed numerically by Berger-Kohn \cite{BKcpam88} (see also Nguyen \cite{NphyD17} for other numerical analysis). In particular, Herrero-Vel\'azquez \cite{HVasnsp92} proved that the blowup dynamic \eqref{eq:prostab} is \textit{generic} in one dimensional case, and they announced the same for higher dimensional case (but never published it).

The method of \cite{BKnon94} and \cite{MZdm97} relies on the understanding of the spectral property of the linearized operator around an expected profile in the \textit{similarity variables} setting. Roughly speaking, the linearized operator possesses a finite number of positive eigenvalues, a null eigenvalue and a negative spectrum; then they proceed in two steps:
\begin{itemize}
\item Reduction of an infinite dimensional problem to a finite dimensional one in the sense that the control of the error reduces to the control of the components corresponding to the positive eigenvalues.  
\item Solving the finite dimensional problem thanks to a classical topological argument based on the index theory.  
\end{itemize}
This general two-step procedure has been extended to various situations such as the case of the complex Ginzgburg-Landau equation by Masmoudi-Zaag \cite{MZjfa08}, Nouaili-Zaag \cite{NZarma18} (see also Zaag \cite{ZAAihn98} for an earlier work); the complex semilinear heat equation with no variational structure by Duong \cite{Darx17}, Nouaili-Zaag \cite{NZcpde15}; non-scaling invariant semilinear heat equations by Ebde-Zaag \cite{EZsema11}, Nguyen-Zaag \cite{NZsns16}, Duong-Nguyen-Zaag \cite{DNZtjm18}. We also mention the work of Tayachi-Zaag \cite{TZnor15,TZpre15} and Ghoul-Nguyen-Zaag \cite{GNZjde17} dealing with a nonlinear heat equation with a double source depending on the solution and its gradient in some critical setting. In \cite{GNZihp18, GNZjde18}, we successfully adapted the method to construct a stable blowup solution for a non variational semilinear parabolic system.\\

As for the present paper, we aim at extending the above mentioned method to construct for problem \eqref{Pb} finite time blowup solutions satisfying some prescribed asymptotic behavior. Although the general idea is the same as for the classical case \eqref{eq:she}, we would like to emphasis that the above mentioned strategy is heavy and its implementation never being straightforward, the context and difficulties are different for each specific problem. As a step forward to better understanding the blowup dynamics for \eqref{Pb}, we obtain the following result.

\begin{theorem}[Type I blowup solutions for \eqref{Pb} with a prescibed behavior] \label{theo:1} Let $m \geq 1$ be an odd integer. There exists initial data $u_0 \in L^\infty(\Rb^N)$ such that the corresponding solution $u(t)$ to equation \eqref{Pb} satisfies\\
$(i)$ $u(t)$ blows up in a finite time $T= T(u_0)$ only at the origin.\\
$(ii)$ (Asymptotic behavior)
\begin{equation}\label{eq:asymtheo}
\sup_{x \in \Rb^N}\left|(T-t)^{\frac{1}{p-1}}u(x,t) - \Phi\left(\frac{x}{\big[(T-t)|\log(T-t)|\big]^{\frac {1}{2m}}}\right) \right| \leq \frac{C}{|\log (T-t)|^\frac{1}{2m}},
\end{equation}
where 
\begin{equation}\label{def:Phipro}
\Phi(\xi) = \kappa\Big(1 + B_{m,p}|\xi|^{2m}\Big)^{-\frac{1}{p-1}},
\end{equation}
with 
\begin{equation}\label{def:Bmpth}
\kappa = (p-1)^{-\frac{1}{p-1}}, \quad B_{m,p}  = (-1)^{m+1}\frac{2(p-1) (2m)!}{p ((4m)! - 2[(2m)!]^2)}.
\end{equation}
$(iii)$ (Final blowup profile) There exists $u^* \in \Cc(\Rb^N \backslash \{0\}, \Rb)$ such that $u(x,t) \to u^*(x)$ as $t \to T$ uniformly on compact subsets of $\Rb^N \backslash \{0\}$, where 
\begin{equation}
u^*(x) \sim \kappa\left(\frac{B_{m,p} |x|^{2m}}{2m |\log |x||} \right)^{-\frac{1}{p-1}} \quad \textup{as} \quad |x| \to 0.
\end{equation}

\end{theorem}

\begin{remark} We believe that such a blowup profile \eqref{def:Phipro} exists for all $m  \in \mathbb{N}^*$. We note that the constant $B_{m,p} < 0$ when $m$ is even (see \eqref{def:Bmpth}), so the profile $\Phi$ blows up on the finite interface $|\xi| = \xi_*= \big(-B_{m,p}\big)^{-\frac{1}{2m}}$. This says that the case when $m$ is even would lead to type II blowup solutions in the sense of \eqref{def:type1}. Although main ideas for a full justification of such a blowup behavior with $m$ even remains the same, the proof would be very delicate and will be addressed in a separate work.
\end{remark}

\begin{remark} The blowup solution described in Theorem \ref{theo:1} is not self-similar in the sense of \eqref{def:selfsol}. Note that in  contrast to blowup solutions of the classical second order semilinear heat equation \eqref{eq:she}, Budd-Galaktionov-Williams \cite{BGWsjam04} through numerical and asymptotic calculations conjectured that there are at least $2\lfloor \frac{m}{2}\rfloor$ nontrivial self-similar blowup solutions to \eqref{Pb}, and that profiles having a single maximum correspond to stable (generic) self-similar blowup solutions.
\end{remark}

\begin{remark} \label{remark:2} The proof of Theorem \ref{theo:1} (in dimension $N = 1$ for simplicity) involves a detailed description of the set of initial data leading to the asymptotic dynamic \eqref{eq:asymtheo}. In particular, our initial data is roughly of the form (see formula \eqref{def:phiAs0d} below)
$$u_0(x) = e^{-s_0} \left[\Phi\left(x e^{s_0/2m}\right) + \frac{A}{s_0^2}\chi\left(2xe^{s_0/2m}, s_0 \right)\sum_{k = 1}^{2m-1}d_k \psi_k(xe^{s_0/2m})\right],$$
where $A$ and $s_0$ are large fixed constants, $\Phi$ is the profile defined by \eqref{def:Phipro}, $\chi$ is some smooth cut-off function, $\big(d_0, \cdots, d_{2m - 1}\big) \in \Rb^{2m}$ are free parameters, $\psi_k, 0 \leq k \leq 2m-1$ are the eigenfunctions of the linearized operator (see Proposition \ref{prop:specL} for a precise definition) corresponding to the positive eigenvalue $\lambda_k = 1 - \frac{k}{2m}$. Through a topological argument, we show that there exists a suitable choice of parameters $\big(d_0, \cdots, d_{2m - 1}\big)$ such that the solution to equation \eqref{Pb} with the initial datum $u_0$ satisfies the conclusion of Theorem \ref{theo:1}. In some sense, our constructed solution is $2(m-1)$-codimension stable in the following sense. The $2m$ components of the linearized solution corresponding to $\lambda_k = 1 - \frac{k}{2m}$ have the exponential growth $e^{\lambda_k s}$. However, the first two modes corresponding to $\lambda_0$ and $\lambda_1$ can be eliminated by means of the time and space translation invariance of the problem. Hence, by fixing $2(m-1)$ directions  $\psi_{2}, \cdots, \psi_{2m - 1}$ and perturbing the remaining components (in $L^\infty$), we still obtain the same asymptotic dynamic \eqref{eq:asymtheo} of the perturbed solution. The proof of $2(m-1)$-codimenison stability would require some Lipschitz regularity of the considered initial data set and it would be addressed separately in another work.    
\end{remark}

\begin{remark}  According to our construction, the asymptotic dynamic \eqref{eq:asymtheo} lies on the center manifold generated by eigenfunctions corresponding to the null eigenvalue $\lambda_{2m} = 0$. Our analysis can be extended to construct for equation \eqref{Pb} a finite time blowup solution having a different asymptotic dynamic from \eqref{eq:asymtheo}. Such solutions particularly have asymptotic dynamics laying on the stable manifold generated by eigenfunctions corresponding to the negative eigenvalue $\lambda_k  = 1 - \frac{k}{2m} < 0$ with $k \geq 2m + 1$. As explained in Remark \ref{remark:2}, the corresponding initial data leading to such solutions would involve $k$ parameters with $k \geq 2m+1$ (consider $N = 1$), so that a topological argument is assigned in order to control the first $k$ components corresponding to the eigenvectors $\psi_j$ for $0 \leq j \leq k -1$. Although the constructive method are similar for all cases, we decide to only deal with  the case of the center manifold, since the proof is the most delicate in the sense that it requires a more refined analysis in the blowup region leading to some logarithmic correction of the blowup variable as shown in \eqref{eq:asymtheo}. 
\end{remark}

\paragraph{Strategy of the proof.} Let us sketch the main ideas of the proof of Theorem \ref{theo:1}. 

\noindent - \underline{\textit{Similarity variables and linearized problem.}} According to the scaling invariance of the problem \eqref{Pb}, we introduce the \textit{similarity variables}
\begin{equation}\label{def:simvars}
u(x,t) = (T-t)^{-\frac{1}{p-1}}w(y,s), \quad y = \frac{x}{(T-t)^\frac{1}{2m}}, \quad s = -\ln(T-t),
\end{equation}
which leads to the new equation
\begin{align}
\partial_s w &= \Am w - \frac{1}{2m}y\cdot \nabla w - \frac{1}{p-1}w + w|w|^{p-1} \nonumber\\
&\equiv \Ls_m w - \frac{p}{p-1}w + w|w|^{p-1}, \label{eq:w}
\end{align}
where the linear operator $\Ls_m$ is given by 
\begin{equation} \label{def:L}
\Ls_m = \Am - \frac{1}{2m}y\cdot \nabla + \textup{Id}.
\end{equation}
Note that the change of variables \eqref{def:simvars} allows us to reduce the finite time blowup problem to a long time behavior one at the cost of an extra scaling term in the new equation \eqref{eq:w}. In the setting \eqref{def:simvars}, proving the asymptotic dynamic \eqref{eq:asymtheo} is equivalent to proving that
\begin{equation}\label{eq:goal}
\sup_{y \in \Rb^N} \left|w(y,s) - \Phi\left(ys^{-\frac{1}{2m}}\right)\right| \to 0 \quad \textup{as} \;\; s \to +\infty.
\end{equation}
It is then natural to linearize equation \eqref{eq:w} around the expected profile $\Phi$ by introducing 
$$q = w - \Phi.$$
which leads to the equation of the form
\begin{equation}\label{eq:q}
\partial_s q = \big(\Ls_m + V(y,s)\big) q + B(q) + R(y,s),
\end{equation}
where $B(q)$ is built to be quadratic, $R$ measures the error generated by $\Phi$ and is uniformly bounded by $\Oc(s^{-1})$, and $V$ is the potential defined as
$$V(y,s) = p\left(\Phi^{p-1} - \kappa^{p-1}\right).$$
Our goal is to construct for equation \eqref{eq:q} a solution $q$ defined for all $(y,s) \in \Rb^N \times [s_0, +\infty)$ such that 
$$\sup_{y \in \Rb^N}|q(y,s)| \to 0 \quad \textup{as}\;\; s \to +\infty.$$

\noindent - \underline{ \textit{Properties of the linearized operator.}} In view of equation \eqref{eq:q}, we see that the nonlinear quadratic and the error term are small and can be negligible in comparison with the linear term. Roughly speaking, the linear part will play an important role in the dynamic of the solution. It is essential to determine the spectrum and corresponding eigenfunctions of both $\Ls_m$ and its adjoint $\Ls_m^*$.  According to \cite{Grsl01}, the spectrum of the linear operator $\Ls_m$ comprises real simple eigenvalues only,
$$\textup{spec}(\Ls_m) = \left\{\lambda_\beta = 1 - \frac{|\beta|}{2m},\; \beta = (\beta_1, \cdots, \beta_n)\in \mathbb{N}^n, \; |\beta| = \beta_1 + \cdots+ \beta_n \right\},$$
and the corresponding eigenfunction $\psi_\beta$ with $|\beta| = n$ is polynomial of order $n$ (see Proposition \ref{prop:specL} below). Moreover, the family of the eigenfunctions $\{\psi_\beta\}_{\beta \in \mathbb{N}^n}$ forms a complete subset in $L^2_\rho(\Rb^N)$ where $\rho$ is some exponentially decaying weight function. 

Depending on the asymptotic behavior of the potential $V$, we observe that
\begin{itemize}
\item Inside the blowup region, $|y| \leq Ks^\frac{1}{2m}$ for some $K$ large, the effect of $V$ is regarded as a perturbation of $\Ls_m$. 
\item Outside the blowup region, $|y| \geq Ks^\frac{1}{2m}$, the full linear part $\Ls_m + V$ behaves like $\Ls_m - \frac{p}{p-1}$, which has a purely negative spectrum. Hence, the control of the solution in this region is simple. 
\end{itemize}

\noindent - \underline{ \textit{Decomposition of the solution and reduction to a finite dimensional problem.}} According to the spectrum of $\Ls_m$, we decompose 
$$q(y,s) = \sum_{|\beta| = 0}^{2m} q_{\beta}(s)\psi_\beta(y) + q_-(y,s),$$
where $q_-$ is the projection of $q$ on the subspace of $\Ls_m$ corresponding to strictly negative eigenvalues. Since the spectrum of the linear part of the equation satisfied by $q_-$ is negative, it is controllable to zero. We would like to notice that we do not use the Feymann-Kac representation\footnote{In \cite{BKnon94} and \cite{MZdm97}, the kernel $\mathcal{K}$ of the heat semigroup associated to the linear operator  $\Ls_1 + V$ is defined through the Feymann-Kac representation
$$\mathcal{K}(s, \sigma, y,x) = e^{(s -\sigma)\Ls_1}(y,x)\int d\mu_{y,x}^{s- \sigma}(\omega)e^{\int_{0}^{s-\sigma}V(\omega(\tau), \sigma + \tau)d\tau},$$ where $e^{t\Ls_1}(y,x)$ is given by Mehler's formula and  $d\mu_{y,x}^t$ is the oscillator measure on the continuous path: $\omega: [0,t] \to \Rb^N$ with $\omega(0) = x$ and $\omega(t) = y$.} as for the case $m = 1$ treated in \cite{MZdm97}, because of its complicated implementation for higher order cases $m \geq 2$. To avoid such a formula, we further decompose 
$$q_-(y,s) = \sum_{|\beta| = 2m + 1}^M q_\beta(s)\psi_\beta + q_{_{M,\bot}}(y,s),$$
for some $M$ large enough (typically $M \ge 4\|V\|_{L^\infty_{y,s}}$). A direct projection yields 
$$ q_\beta' = \left(1 - \frac{|\beta|}{2m}\right)q_\beta + \Oc\left(s^{-\frac{|\beta| + 2}{2m}}\right), \quad |\beta| = 2m+1, \cdots, M,$$
from which we obtain the rough bound $|\theta_\beta(s)| = \Oc(s^{-\frac{|\beta| + 2}{2m}})$ for $|\beta| = 2m+1,\cdots, M$. For the infinite part $q_{_{M,\bot}}$, we explore the properties of the semigroup $e^{s\Ls_m}$ and a standard Gronwall inequality to close the estimate for this part. 

The control of the null mode $q_{2m}$ is delicate since the potential has in some sense a critical size in our analysis. In particular, we need a careful refinement of the asymptotic behavior of $V$ to derive the sharp ODE 
$$q_{\beta}' = -\frac{2}{s}q_{\beta} + \Oc\left(\frac{1}{s^3}\right) \quad \textup{for}\;\; |\beta| = 2m,$$
which shows a negative spectrum after changing the variable $\tau = \ln s$, hence, the rough bound $|q_{\beta}(s)| = \Oc\left(\frac{\log s}{s^2}\right)$ for $|\beta| = 2m$. Here the precise value of $B_{m,p}$ given in \eqref{def:Bmpth} is crucial for many algebraic identities to derive this sharp ODE.  At this stage, we reduce the infinite dimensional problem to a finite dimensional one in the sense that it remains to control a finite number of positive modes $q_\beta$ for $|\beta| \leq 2m - 1$. This is done through a classical topological argument based on the index theory. \\

The rest of the paper is organized as follows: In Section \ref{sec:2} we recall basic spectral properties of the linearized operator $\Ls$ and its adjoint $\Ls^*$, then we perform a formal spectral analysis to derive an approximate blowup profile served for our analysis later. In Section \ref{sec:3} we give all arguments for the proof of Theorem \ref{theo:1} without going to technical details (the reader who is not interested in the technical details can stop at this section). Section \ref{sec:4} is the hearth of our analysis: it is devoted to the study of the dynamic of the linearized problem from which we reduce the problem to a finite dimensional one.\\

\section{A formal approach via spectral analysis.}\label{sec:2}
In this section we first recall basic spectral properties of the linearized operator, then present a formal approach based on a spectral analysis to derive the blowup profile given in Theorem \ref{theo:1}.

\subsection{Spectral properties of the linearized operator.}
In this subsection, we recall from \cite{Grsl01} the basic spectral properties of the linear operator $\Ls_m$ and its adjoint $\Ls_m^*$. The case $m = 1$ is well known, since we can rewrite
$$\Ls_1 = \frac{1}{\rho_1}\nabla.(\rho_1 \nabla) + \textup{Id} \quad \textup{with}\quad \rho_1(y) = (4\pi)^{-\frac{N}{2}}e^{-\frac{|y|^2}{4}},$$
which is a self-adjoint operator in the weighted Hilbert space $L^2(e^{-|y|^2/4}dy)$ with the domain $\Dc(\Ls) = H^2(e^{-|y|^2/4}dy)$. It has a real discrete spectrum and the corresponding eigenfunctions are derived from Hermite polynomials.

For $m \geq 2$, the operator $\Ls_m$ is not symmetric and does not admit a self-adjoint extension. Then we denote $\Ls^*_m$ the formal adjoint of $\Ls_m$ as 
\begin{equation}\label{def:L*}
\Ls^*_mf = \Am f + \frac{1}{2m}\nabla.(yf) + f,
\end{equation}
in the sense that
\begin{equation}
\big<\Ls_m f, g\big> = \big<f, \Ls^*_m g\big> \quad \textup{for}\quad (f,g)  \in \Dc(\Ls_m) \times \Dc(\Ls^*_m).
\end{equation}
From \cite{Ebook69} and \cite{Fumn77}, we know that the following elliptic equation 
\begin{equation}\label{eq:F}
\Am F + \frac{1}{2m}\nabla.(yF) = 0 \quad \textup{in}\;\; \Rb^N \quad \textup{with} \;\; \int_{\Rb^N}F(y)dy = 1.
\end{equation}
has a unique radial solution given by the explicit formula 
\begin{equation}\label{def:F}
F(y) = (2\pi)^{-\frac N2}\int_0^\infty e^{-s^{2m}}s^{\frac N2}J_{\frac{N-2}2}(s|y|) ds,
\end{equation}
with $J_\nu$ being the Bessel function. In particular, $F$ satisfies the estimate 
\begin{equation}\label{est:F}
|F(y)| < De^{-d|y|^\nu} \quad \textup{in}\;\; \Rb^N, \;\; \textup{where}\;\; \nu = \frac{2m}{2m-1},
\end{equation}
for some positive constants $D$ and $d$ depending on $m$ and $N$.\\
We introduce the weight functions
\begin{equation}
\rho^* = \rho^{-1}, \quad \rho(y) = e^{-a|y|^\nu},
\end{equation}
where $a = a(m,N) \in (0, d]$ is a small constant, $d$ and $\nu$ are introduced in \eqref{est:F}. 

\begin{proposition}[Spectral properties of $\Ls_m$ and $\Ls^*_m$, \cite{Grsl01}] \label{prop:specL} Let $m \in \mathbb{N}^*$, we have\\
$(i)$ $\Ls_m: H^{2m}_{\rho}(\Rb^N) \to L^2_{\rho}(\Rb^N)$ is a bounded linear operator with the spectrum
\begin{equation}\label{eq:specL}
\textup{spec}(\Ls_m) = \left\{\lambda_\beta = 1 - \frac{|\beta|}{2m}, \quad \beta = (\beta_1, \cdots, \beta_N) \in \mathbb{N}^N,\quad |\beta| = 0, 1, 2, \cdots\right\}.
\end{equation}
The set of eigenfunctions $\{\psi_\beta\}_{|\beta| \geq 0}$ is complete in $L^2_\rho(\Rb^N)$, where $\psi_\beta$ has the separable decomposition
\begin{equation}\label{def:psibeta}
\psi_\beta(y) = \psi_{\beta_1}(y_1) \cdots \psi_{\beta_N}(y_N) \quad \textup{with}\quad \psi_k(\xi) =\frac{1}{\sqrt{k!}}\sum_{j = 0}^{\left \lfloor \frac{k}{2m} \right\rfloor}\frac{(-1)^{jm}}{j!}\partial_\xi^{2jm}\xi^k. 
\end{equation} 

\noindent $(ii)$ $\Ls^*_m: H^{2m}_{\rho^*}(\Rb^N) \to L^2_{\rho^*}(\Rb^N)$ is a bounded linear operator with the spectrum
\begin{equation}\label{eq:spec*}
\textup{spec}(\Ls^*_m) = \left\{\lambda_\beta^* = 1 - \frac{|\beta|}{2m}, \quad \beta = (\beta_1, \cdots, \beta_N) \in \mathbb{N}^N,\quad |\beta| = 0, 1, 2, \cdots\right\}.
\end{equation}
The set of eigenfunctions $\{\psi_\beta^*\}_{|\beta| \geq 0}$ is complete in $L^2_{\rho^*}(\Rb^N)$, where $\psi_\beta^*$'s are given by  
\begin{equation}\label{def:psibeta*}
\psi^*_\beta(y) = \frac{(-1)^{|\beta|}}{\sqrt{\beta !}}\partial^\beta_y F(y),
\end{equation}
where $\beta! = \beta_1! \cdots \beta_N!$,  $\;\partial^\beta_y = \frac{\partial^{\beta_1}}{\partial y_1} \cdots \frac{\partial^{\beta_N}}{\partial y_N}$ and the function $F$ is defined by \eqref{def:F}.\\

\noindent $(iii)$ (Orthogonality)
\begin{equation}\label{eq:orths}
\left<\psi_\beta, \psi_\gamma^*\right> = \delta_{\beta,\gamma},
\end{equation}
where $\delta_{\beta, \gamma}$ is the Kronecker delta.
\end{proposition}
\begin{remark} We note from the orthogonality \eqref{eq:orths} and the definition of $\psi_k$ that for all polynomials $P_n(y)$ of degree $n < |\gamma|$, we have $\big<P_n, \psi^*_\gamma\big> = 0$.

For $N = 1$, we have 
\begin{align*}
\psi_k(y) &= a_k y^k \quad \textup{for} \;\; 0 \leq k < 2m, \quad a_k = \frac{1}{\sqrt{k!}},\\
\psi_{2mj}(y)& = \sum_{i = 0}^j c_{2mi} y^{2m i} \;\; \textup{and}\quad y^{2m j} = \sum_{i = 0}^j c'_{2mi} \psi_{2mi}(y) \quad \textup{for} \quad j \in \mathbb{N}. 
\end{align*}
\end{remark}

\bigskip

\noindent We end this subsection by recalling basic properties of the semigroup $e^{s\Ls_m}$ for $s > 0$. 
\begin{lemma}[Properties of the semigroup $e^{s\Ls_m}$]  \label{lemm:semigLm} The kernel of the semigroup $e^{s\Ls_m}$ is given by 
\begin{equation}\label{def:semigroupLm}
e^{s\Ls_m}(y,x) = \frac{e^s}{\big[(1 - e^{-s})\big]^\frac{N}{2m}}F\left(\frac{ye^{-\frac{s}{2m}} - x}{\big[2m(1 - e^{-s})\big]^\frac{1}{2m}} \right), \quad \forall s > 0,
\end{equation}
where $F$ is defined as in \eqref{def:F}. The action of $e^{s\Ls_m}$ is defined by 
\begin{equation}\label{def:actSemigLm}
e^{s\Ls_m}g(y) = \int_{\Rb^N}e^{s\Ls_m}(y,x)g(x)dx.
\end{equation}
We also have the following estimates:\\
$(i)\;$ $\|e^{s\Ls_m}g\|_{L^\infty} \leq Ce^s\|g\|_{L^\infty}$ for all $g \in L^\infty(\Rb^N)$.\\
$(ii)\,$ $\|e^{s\Ls_m} \textup{div}(g)\|_{L^\infty} \leq Ce^s(1 - e^{-s})^{-\frac{1}{2m}}\|g\|_{L^\infty}$ for all $g \in L^\infty(\Rb^N)$.\\
$(iii)\,$ If $|f(x)| \leq \eta (1 + |x|^{M+1})$ for all $x \in \Rb^N$, then 
\begin{equation}
\left|e^{s\Ls_m}\Pi_{_{M,\bot}}f(y)\right| \leq C\eta e^{-\frac{Ms}{2m}}(1 + |y|^{M+1}), \quad \forall y \in \Rb^N,
\end{equation}
where $\Pi_{_{M,\bot}}$ is defined as in \eqref{def:PiMbot}.
\end{lemma}
\begin{proof} The formula \eqref{def:semigroupLm} can be verified by a direct computation thanks to equation \eqref{eq:F}. The estimates $(i)$-$(iii)$ are straightforward from the definitions \eqref{def:semigroupLm} and \eqref{def:actSemigLm}.
\end{proof}

\subsection{Approximate blowup profile.}\label{sec:22}
In this subsection we recall the formal approach of \cite{GALnon09} (see also \cite{BGWsjam04}) to figure out an appropriate blowup profile for our analysis later. This approach had been used in several problems involving the second order Laplacian, see for example \cite{BKnon94}, \cite{TZpre15}, \cite{GNZjde17,  GNZihp18, GNZjde18}. The argument relies on the basis of the known spectral properties of the rescaled operator $\Ls_m$ and its adjoint $\Ls^*_m$ given in the previous subsection. For simplicity, we consider the one dimensional case and symmetric positive solutions. 

Let us introduce 
$$\bar w = w - \kappa,$$
where $\kappa = (p-1)^{-\frac{1}{p-1}}$ is the constant equilibrium to equation \eqref{eq:w}. This yields the following perturbed equation
\begin{equation}\label{eq:wbar}
\partial_s \bar w = \Ls_m \bar w + \frac{p}{2\kappa}\bar{w}^2 + R(\bar w),
\end{equation}
where $|R(\bar w)| \leq C|\bar w|^3$ for $\bar w \ll 1$.\\
From Proposition \ref{prop:specL}, we know that $\psi_n$ with $n \geq 2m + 1$ correspond to negative eigenvalues of $\Ls_m$. Therefore, we may consider 
\begin{equation}\label{eq:aswbar}
\bar w(y,s) =  \sum_{i = 0}^{m}\bar w_{2i}(s)\psi_{2i}(y),
\end{equation}
where $\sum_{i = 0}^m|\bar w_{2i}(s)| \to 0$ as $s \to +\infty$. Plugging this ansatz to equation \eqref{eq:wbar}, taking the scalar product with $\psi_{2j}^*$ and using Proposition \ref{prop:specL}, we find that for $0 \leq j \leq m$,
\begin{equation}\label{eq:w2j}
\bar w'_{2j}(s) = \left(1 - \frac{j}{m}\right)w_{2j}(s) + \frac{p}{2\kappa}\left< \left(\sum_{j = 0}^m\bar w_{2i}(s)\psi_{2i}(y)\right)^2, \psi^*_{2j}\right> + \Oc(\sum_{i = 0}^m|\bar w_{2i}(s)|^3).
\end{equation} 
Assuming that $\bar w_{2m}$ is dominant, i.e.
\begin{equation}\label{ass:w2m}
\textup{for}\;\; 0 \leq j \leq m- 1, \quad |\bar w_{2j}(s)| \ll |\bar w_{2m}(s)| \quad \textup{as}\;\; s \to +\infty,
\end{equation}
then system \eqref{eq:w2j} reduces to 
\begin{align*}
\bar w'_{2j}(s) &= \left(1 - \frac{j}{m}\right)w_{2j}(s) + \Oc(|\bar w_{2m}(s)|^2) \quad \textup{for}\quad 0 \leq j \leq m - 1, \\ 
\bar w'_{2m}(s) &= \frac{p}{2\kappa}\bar \mu_{2m}\bar w_{2m}^2 + o(|\bar w_{2m}(s)|^2), \quad \textup{where}\quad  \bar \mu_{2m} = \left<\psi_{2m}^2, \psi_{2m}^*\right>.
\end{align*}
Solving this system yields
$$\bar w_{2m}(s) = - \frac{2\kappa}{p \bar \mu_{2m}}\frac 1s + \Oc\left(\frac{\ln^2 s}{s}\right), \quad \sum_{i = 0}^{m-1} |\bar w_{2i}(s)| = \Oc\left(\frac 1{s^{2}}\right) \quad \textup{as}\;\; s \to +\infty,$$
which is in agreement with the assumption \eqref{ass:w2m}. \\
Hence, from \eqref{eq:aswbar} and the definition of $\psi_{2m}$, we derive the following asymptotic behavior 
\begin{equation}\label{eq:expwfor}
w(y,s) = \kappa  - \frac{2\kappa}{p \bar \mu_{2m} s}\left(\frac{y^{2m} + (-1)^m (2m)!}{\sqrt{(2m)!}}\right) + \Oc\left(\frac{\ln s}{s^2}\right),
\end{equation}
where the convergence takes place in $L^2_{\rho}(\Rb)$ as well as uniformly in compact sets by standard parabolic regularity. 

The expansion \eqref{eq:expwfor} provides a relevant variable for blowup, namely $z = y s^{-\frac{1}{2m}}$, that governs the behavior in the intermediate region. In particular, we try to search formally a solution $w$ of equation \eqref{eq:w} of the form
\begin{equation}\label{eq:expqfor2}
w(y,s) = \Phi(z)  + \frac{(-1)^{m+1}2\kappa\sqrt{(2m)!}}{p \bar \mu_{2m} s} + \Oc\left(\frac{1}{s^{1 + \epsilon}}\right),
\end{equation}
for some $\epsilon > 0$, with the boundary condition $\Phi(0) = \kappa$.\\
Plugging this ansatz to equation \eqref{eq:w} and comparing the leading order terms, we arrive at 
\begin{equation}\label{eq:Phi}
-\frac{z}{2m}\Phi' - \frac{1}{p-1}\Phi + \Phi^p = 0, \quad \Phi(0) = \kappa.
\end{equation}
Solving this ODE yields
$$\Phi(z) = \kappa\Big(1 + B_{m,p}z^{2m}\Big)^{-\frac{1}{p-1}} \quad \textup{for some $B_{m,p} \in \Rb$}.$$
By matching expansions \eqref{eq:expwfor} and \eqref{eq:expqfor2}, we find that 
\begin{equation*}
B_{m,p} = \frac{2(p-1)}{p \bar \mu_{2m} \sqrt{(2m)!}}, \quad \bar \mu_{2m} = \left<\psi_{2m}^2, \psi^*_{2m}\right>.
\end{equation*}
Let us compute $\bar \mu_{2m}$ in the following. Using \eqref{def:psibeta}, we write
\begin{align*}
\psi_{2m}(y) &= \frac{1}{\sqrt{(2m)!}}\big(y^{2m} + (-1)^m (2m)!\big),\\
\psi_{4m}(y) &= \frac{1}{\sqrt{(4m)!}}\left(y^{4m} + (-1)^m \frac{(4m)!}{(2m)!} y^{2m} + (4m)!\right),\\
\psi^2_{2m}(y) &= \frac{1}{(2m)!}\left[\sqrt{(4m)!}\psi_{4m} + (-1)^{m+1}\sqrt{(2m)!}\left(\frac{(4m)!}{(2m)!} -  2(2m)!\right)\psi_{2m} + c_0(m)\psi_0 \right].
\end{align*}
Using the orthogonality relation \eqref{eq:orths}, the definition \eqref{def:psibeta*} of $\psi_{2m}^*$ and the fact that $\int_{\Rb} F(y)dy = 1$, we compute by an integration by parts, 
\begin{align}
\bar \mu_{2m} &= \frac{(-1)^{m+1}}{\sqrt{(2m)!}}\left(\frac{(4m)!}{(2m)!} -  2(2m)!\right)\int_\Rb \psi_{2m}(y) \psi^*_{2m}(y)dy \nonumber\\
& = \frac{(-1)^{m+1}}{\sqrt{(2m)!}}\left(\frac{(4m)!}{\big[(2m)!\big]^2} -  2\right)\int_\Rb (y^{2m} + c_{2m,0})\partial_y^{2m}F(y)dy\nonumber\\
& =  (-1)^{m+1}\frac{(2m)}{\sqrt{(2m)!}}\left(\frac{(4m)!}{\big[(2m)!\big]^2} -  2\right)\int_\Rb F(y)dy\nonumber\\
&= (-1)^{m+1}\sqrt{(2m)!}\left(\frac{(4m)!}{[(2m)!]^2} - 2\right).\label{def:mu2m}
\end{align}
In conclusion, we have derived the following candidate for the blowup profile in the similarity variables:
\begin{equation}\label{def:appPro}
w(y,s) \sim \varphi(y,s):= \kappa\left(1 + B_{m,p} \frac{y^{2m}}{s}\right)^{-\frac{1}{p-1}} +  \frac{A_{m,p}}{s},
\end{equation}
where 
\begin{equation}\label{def:Bpm}
B_{m,p} = \frac{2(p-1)}{p \bar \mu_{2m} \sqrt{(2m)!}}, \quad  A_{m,p} = \frac{(-1)^{m+1}2\kappa\sqrt{(2m)!}}{p \bar \mu_{2m}}.
\end{equation}
Since we want the profile $\varphi$ bounded, this requests $B_{m,p} > 0$ or $\bar{\mu}_{2m} > 0$, which only happens for $m$ odd.

\section{Proof of Theorem \ref{theo:1} without technical details.}\label{sec:3}
In this section we give all arguments of the proof of Theorem \ref{theo:1}. We only deal with the one dimensional case $N = 1$ for simplicity since the analysis for the higher dimensional cases $N \geq 2$ is exactly the same up to some complicated calculation of the projection of \eqref{eq:q} on the eigenspaces of $\Ls_m$. The proof of Theorem \ref{theo:1} is completed in three parts:
\begin{itemize}
\item In the first part we formulate the problem by linearizing the rescaled equation \eqref{eq:w} around the approximate profile $\varphi$ given by \eqref{def:appPro}. We also introduce a shrinking set in which the constructed solution of the linearized equation is trapped. 
\item In the second part we exhibit an explicit formula of the initial data and show that the corresponding solution belongs to the shrinking set. In particular, the reader can find how to reduce the problem to a finite dimensional one (all technical details will be left to the next section) and the use of a topological argument based on index theory to conclude.
\item The last part is devoted to the proof of items $(i)$ and $(iii)$ of Theorem \ref{theo:1}. Although the argument of the proof is almost the same as for the classical case $m = 1$, we would like to sketch the main ideas for the reader convenience. 
\end{itemize} 

\subsection{Formulation of the problem.}

According to the formal analysis given in Section \ref{sec:22}, we introduce 
\begin{equation}\label{def:qwphi}
q(y,s) = w(y,s) - \varphi(y,s),
\end{equation}
and write from \eqref{eq:w} the equation driving by $q$,
\begin{equation}
\partial_s q = \big(\Ls_m + V(y,s)\big) q + B(q) + R(y,s),
\end{equation} 
where $\Ls_m$ is the linearized operator defined by \eqref{def:L} and 
\begin{align}
V(y,s) &= p\left(\varphi^{p-1} - \kappa^{p-1}\right),\label{def:V}\\
B(q) &= (q + \varphi)|q + \varphi|^{p-1} - \varphi^p -  p \varphi^{p-1}q,\label{def:B}\\
R(y,s) &= -\partial_s \varphi + \Am \varphi - \frac{1}{2m}y\cdot \nabla \varphi - \frac{\varphi}{p-1} + \varphi^p.\label{def:R}
\end{align}
Let $\chi_0 \in \Cc_0^\infty(\Rb_+)$ be a cut-off function with $\textup{supp}(\chi_0) \subset [0,2]$ and $\chi_0 \equiv 1$ on $[0,1]$. We introduce 
\begin{equation}\label{def:chi}
\chi(y,s) = \chi_0\left(\frac{|y|}{Ks^\frac{1}{2m}}\right) \quad \textup{with}\;\; K \gg 1 \;\; \textup{fixed},
\end{equation}
and define
\begin{equation}\label{def:qe}
q_e(y,s) =  q(y,s)(1 - \chi(y,s)).
\end{equation}
We decompose 
\begin{equation}\label{eq:decomq}
q(y,s) = \sum_{k = 0}^M q_k(s)\psi_k(y) + q_{_{M,\bot}}(y,s),
\end{equation}
where
\begin{itemize}
\item $q_k = \Pi_{k}(q)$ is the projection of $q$ on the eigenmode corresponding to eigenvalue $\lambda_k = 1 - \frac{k}{2m}$, defined as 
\begin{equation}
q_k(s) = \big<q, \psi_k^*\big> \quad \textup{with \, $\psi^*_k$ \, being introduced in \eqref{def:psibeta*}.}
\end{equation}
\item $q_{_{M,\bot}} = \Pi_{_{M, \bot}}(q)$ is called the infinite dimensional part of $q$, where $\Pi_{_{M, \bot}}(q)$ is the projection of $q$ on the eigensubspace where the spectrum of $\Ls_m$ is lower than $\frac{1 - M}{2m}$. Note that we have the orthogonality 
\begin{equation}\label{def:PiMbot}
\big<q_{_{M,\bot}}, \psi_k^*\big> = 0 \quad \textup{for}\; k \leq M.
\end{equation}
\item  $M$ is typically a large constant 
\begin{equation}\label{def:M}
M = 4mN \quad \textup{with}\;\; N \in \mathbb{N}^*, \; N \geq \|V\|_{L^\infty_{y,s}}.
\end{equation}
which allows us to successfully apply a standard Gronwall's inequality to the control of the infinite dimensional part $q_{_{M,\bot}}$. 
\end{itemize}
We aim at constructing for equation \eqref{eq:q} a global in time solution $q$ such that 
\begin{equation}\label{est:goalqLinf}
\|q(s)\|_{L^\infty(\Rb^N)} \to 0 \quad \textup{as} \quad s \to +\infty.
\end{equation}
According to the decomposition \eqref{eq:decomq}, it is enough to show that there exists a solution $q$ belonging to the following set.
\begin{definition}[Shrinking set to trap solutions] \label{def:VA} For each $A > 0$, for each $s > 0$, we denote $\Vc_A(s)$ the set of all functions $q(y,s)$ in $L^\infty(\Rb)$ such that 
\begin{align*}
&|q_k(s)| \leq \frac{A}{s^2} \quad  \textup{for}\;\; 0 \leq k \leq 2m - 1, \qquad \quad  |q_{2m}(s)| \leq \frac{A^2 \log s}{s^2},\\
&\quad |q_k(s)| \leq A^{k} s^{-\frac{k + 1}{2m}}  \quad \textup{for}\;\; 2m + 1 \leq k \leq M, \qquad \quad \|q_e(s)\|_{L^\infty(\Rb)} \leq A^{M+2}s^{-\frac{1}{2m}},\\
&\qquad \forall y \in \Rb, \quad |q_{_{M,\bot}}(y,s)| \leq A^{M+1}s^{-\frac{M+2}{2m}}(|y|^{M+1} + 1),
\end{align*}
where $q_k$, $q_{_{M,\bot}}$ and $q_e$ are defined as in the decomposition \eqref{eq:decomq} and  \eqref{def:qe}.
\end{definition}
\begin{remark}\label{remark:1} By definition, we see that if $q(s) \in \Vc_A(s)$, then the following bound holds
\begin{align*}
|q(y,s)| &\leq CA^{M+1}\left(\sum_{k = 0}^{M+1} (|y|^{k}  +1)s^{-\frac{k + 1}{2m}} \right)\mathbf{1}_{\{|y| \leq 2Ks^\frac{1}{2m}\}} + \|q_e\|_{L^\infty(\Rb^N)}\\
& \qquad  \leq CA^{M+2} s^{-\frac{1}{2m}} \quad \textup{for all} \quad y \in \Rb. 
\end{align*}
Hence, our goal \eqref{est:goalqLinf} reduces to constructing for equation \eqref{eq:q} a solution $q(s)$ belonging to the shrinking set $\Vc_A(s)$ for all $s \in [s_0, +\infty)$. 
\end{remark}

\subsection{Existence of solutions trapped in $\Vc_A$.} 
In this step we aim at proving that there actually exists initial data $\phi(y) = q(y,s_0)$ such that the corresponding solution $q(y,s)$ to \eqref{eq:q} belongs to the shrinking set $\Vc_A(s)$. Given $A \geq 1$ and $s_0 \geq e$, we consider initial data of the form 
\begin{equation}\label{def:phiAs0d}
\phi_{A, s_0, \mathbf{d}}(y) = \frac{A}{s_0^2}\chi(2y, s_0)\sum_{k = 0}^{2m-1} d_k \psi_k(y),
\end{equation}
where $\mathbf{d} = (d_1, \cdots, d_{2m - 1}) \in \Rb^{2m}$ are real parameters to be fixed later, $\chi$ is introduced in \eqref{def:chi}. In particular, the initial data \eqref{def:phiAs0d} belongs to $\Vc_A(s_0)$ as shown in the following proposition.
\begin{proposition}[Properties of initial data \eqref{def:phiAs0d}] \label{prop:initdata} For each $A \gg 1$, there exist $s_0 = s_0(A) \gg 1$ and a cuboid $\Dc_{s_0} \subset [-A, A]^{2m}$ such that for all $\big(d_0, \cdots, d_{2m-1}\big) \in \Dc_{s_0}$, the following properties hold: 
\begin{itemize}
\item[(i)] The initial data $\phi_{A, s_0, \mathbf{d}}$ defined in \eqref{def:phiAs0d} belongs to  $\Vc_A(s_0)$ with strict inequalities except for the positive modes $\phi_k$ with $0 \leq k \leq 2m - 1$, where $\phi_k = \Pi_{k} (\phi_{A, s_0, \mathbf{d}})$.
\item[(ii)] The map $\Gamma: \Dc_{s_0} \to \Rb^{2m}$, defined as $\Gamma(d_0, \cdots, d_{2m-1}) = \big(\phi_0, \cdots, \phi_{2m- 1})$, is linear, one to one from $\Dc_{s_0}$ to $\big[-As_0^{-2}, As_0^{-2}\big]^{2m}$, and maps $\partial \Dc_{s_0}$ into $\partial \Big( \big[-As_0^{-2}, As_0^{-2}\big]^{2m}\big)$. Moreover, the degree of $\Gamma$ on the boundary is different from zero.
\end{itemize} 
\end{proposition}
\begin{proof} The proof directly follows from the definition of the projection $\Pi_k(\phi_{A, s_0, \mathbf{d}})$ and it is straightforward, so we omit it here.
\end{proof}

Starting with initial data $\phi_{A, s_0, \mathbf{d}}$ belonging to $\Vc_A(s_0)$, we claim that we can fine-tune the parameters $\big(d_0, \cdots, d_{2m-1}\big) \in \Dc_{s_0}$ such that the corresponding solution $q(s)$ to equation \eqref{eq:q} stays in $\Vc_A(s)$ for all $s \geq s_0$. More precisely, we claim the following.
\begin{proposition}[Existence of solutions trapped in $\Vc_A(s)$] \label{prop:Exist} There exist $A \gg 1$, $s_0 = s_0(A) \gg 1$ and $\big(d_0, \cdots, d_{2m-1}\big) \in \Dc_{s_0}$ such that if $q(s)$ is the solution to equation \eqref{eq:q} with initial data \eqref{def:phiAs0d}, then $q(s) \in \Vc_A(s)$ for all $s \geq s_0(A)$. 
\end{proposition}
\begin{proof} From the local Cauchy problem of \eqref{Pb} in $L^\infty(\Rb)$, we see that for each initial data $\phi_{A, s_0, \mathbf{d}} \in \Vc_A(s_0)$, equation \eqref{eq:q} has a unique solution $q(s) \in \Vc_A(s)$ for all $s \in [s_0, s_*)$ with $s_* = s_*(\mathbf{d})$. If $s_* = +\infty$ for some $\mathbf{d} = (d_0, \cdots, d_{2m - 1}) \in \Dc_{s_0}$, we are done. Otherwise, we proceed by contradiction and assume that $s_*(\mathbf{d}) < +\infty$ for all $\mathbf{d} \in \Dc_{s_0}$. By continuity and the definition of $s_*$, we remark that $q(s_*) \in \partial \Vc_A(s_*)$.  We claim the following.
\begin{proposition}[Finite dimensional reduction] \label{prop:redu} There exist $A \gg 1$, $s_0 = s_0(A) \gg 1$ and $\big(d_0, \cdots, d_{2m-1}\big) \in \Dc_{s_0}$ such that the following properties hold: If $q(s)$ is the solution to equation \eqref{eq:q} with initial data \eqref{def:phiAs0d} and $q(s) \in \Vc_A(s)$ for all $s \in [s_0, s_1]$ and $q(s_1) \in \partial \Vc_A(s_1)$, then 
\begin{itemize}
\item[(i)] \textup{(Reduction to a finite dimensional problem)} $\;\Big(q_0, \cdots, q_{2m - 1}\Big)(s_1) \in \partial \left([-As_1^{-2}, As_1^{-2}]^{2m}\right)$.
\item[(ii)] \textup{(Transversality)} There exists $\mu_0 > 0$ such that  $q(s_1  + \mu) \not \in \Vc_A(s_1 + \mu)$ for all $\mu \in (0, \mu_0)$.
\end{itemize} 
\end{proposition}

\noindent Let us postpone the proof of Proposition \ref{prop:redu} to the next section and continue our argument. From $(i)$ of Proposition \ref{prop:redu}, we obtain $\big(q_0, \cdots, q_{2m - 1} \big)(s_*)\in \partial \left([-As_*^{-2}, As_*^{-2}]^{2m}\right)$, and the following mapping is well defined
\begin{align*}
\Theta: \Dc_{s_0} \; &\to \; \partial \left([-1,1]^{2m}\right) \\
 (d_0, \cdots, d_{2m - 1}) \; &\mapsto \; \frac{s_*^2}{A}\big(q_0, \cdots, q_{2m - 1}\big)(s_*).
\end{align*} 
From the transversality given in item $(ii)$ of Proposition \ref{prop:redu}, $\big(q_0, \cdots, q_{2m - 1}\big)$ actually crosses its boundary at $s = s_*$, resulting in the continuity of $s_*$ and $\Theta$. 
Applying again the transversality, we see that if $(d_0, \cdots, d_{2m - 1}) \in \partial \Dc_{s_0}$, then $q(s)$ leaves $\Vc_A(s)$ at $s = s_0$, thus, $s_* = s_0$ and $\Theta_{|\partial \Dc_{s_0}} = \Gamma$, the map defined in item $(ii)$ of Proposition \ref{prop:initdata}. Using that item, we see that the degree of $\Theta$  is not zero. Since $\Theta$ is continuous, this is a contradiction. This concludes the proof of Proposition \ref{prop:Exist}, assuming that Proposition \ref{prop:redu} holds.
\end{proof}

\subsection{Conclusion of Theorem \ref{theo:1}.} From Proposition \ref{prop:Exist}, there exists a solution $q$ to equation \eqref{eq:q} such that $q(s) \in \Vc_A(s)$ for all $s \geq s_0$. From Remark \ref{remark:1}, we deduce that $\|q(s)\|_{L^\infty(\Rb)} \leq C(A)s^{-\frac{1}{2m}}$ for all $s \geq s_0$. The conclusion  of item $(ii)$ then follows from \eqref{def:simvars} and \eqref{def:qwphi}.  Item $(i)$ of Theorem \ref{theo:1} is just a direct consequence of items $(ii)$ and $(iii)$. Because the proof of item $(iii)$ is similar to the classical case $m = 1$, we only sketch the main ideas for the reader's convenience. The existence of the final blowup profile $u^* \in \Cc(\Rb \setminus \{0\})$ follows from the technique of Merle \cite{Mercpam92}. Here we focus on a precise description of the final blowup profile $u^*$ in a neighborhood of the singularity. To do so, we follow the technique of Herrero-Vel\'azquez \cite{HVaihn93} (see also Bebernes-Bricher \cite{BBsima92}, Zaag \cite{ZAAihn98} for a similar approach) by introducing the  auxiliary function
\begin{equation}\label{def:hxi}
h(\xi, \tau; x_0) = (T-t_0(x_0))^\frac{1}{p-1}u(x,t),
\end{equation}
where 
\begin{equation*}
\xi = \frac{x - x_0}{[T - t_0(x_0)]^\frac{1}{2m}}, \quad \tau = \frac{t - t_0(x_0)}{T - t_0(x_0)},
\end{equation*}
and $t_0(x_0)$  is uniquely determined by 
\begin{equation}\label{def:xoto}
|x_0| = K\big[(T-t_0(x_0))|\log(T-t_0(x_0))|\big]^\frac{1}{2m}, \quad K \gg 1\; \textup{fixed}.
\end{equation}
We note that $h(\xi,\tau; x_0)$ is also a solution to \eqref{Pb} because of the invariance of \eqref{Pb} under dilations. From \eqref{eq:asymtheo}, we have
$$\sup_{|\xi| \leq |\log(T-t_0(x_0))|^\frac{1}{4m}} \left|h(\xi, 0, x_0) - \Phi(K)\right| \leq \frac{C}{(T - t_0(x_0))^\frac{1}{2m}} \to 0 \quad \textup{as}\; |x_0| \to 0.$$
Let $\hat h_K(\tau)$ be the solution to \eqref{Pb} with the constant initial datum $\Phi(K)$, defined as
$$\hat{h}_K(\tau) = \kappa\left(1 - \tau + B_{m,p}K^{2m}\right)^{-\frac{1}{p-1}}, \quad \tau \in [0,1).$$
By the continuity with respect to initial data for equation \eqref{Pb}, one can show that
$$\sup_{|\xi| \leq |\log(T-t_0(x_0))|^\frac{1}{4m}, 0 \leq \tau < 1} \left|h(\xi, \tau; x_0 ) - \hat h_K(\tau)\right| \leq \epsilon(x_0) \to 0, \quad \textup{as}\;\; |x_0| \to 0.$$ 
Passing to the limit $\tau \to 1$ yields
$$u^*(x_0) = (T - t_0(x_0))^{-\frac{1}{p-1}}\lim_{\tau \to 1} h(0, \tau; x_0) \sim (T -t_0(x_0))^{-\frac{1}{p-1}}\hat h_K(1).$$
From the definition \eqref{def:xoto}, we compute 
$$\big|\log (T-t_0(x_0))\big| \sim 2m \big|\log |x_0|\big|, \quad T - t_0(x_0) \sim \frac{|x_0|^{2m}}{2m K^{2m}|\log |x_0||} \quad \textup{as} \quad |x_0| \to 0,$$
from which we obtain the asymptotic behavior
$$u^*(x_0) \sim \kappa\left(\frac{B_{m,p}|x_0|^{2m}}{2m|\log |x_0||}\right)^{-\frac{1}{p-1}} \quad \textup{as} \quad |x_0| \to 0.$$
This concludes the proof of Theorem \ref{theo:1}, assuming that Proposition \ref{prop:redu} holds.

\section{Reduction to a finite dimensional problem.} \label{sec:4}
This section is the central part in our analysis where we give all details of the proof of Proposition \ref{prop:redu}, completing hence the proof of Theorem \ref{theo:1}. The essential idea is to project equation \eqref{eq:q} onto different components according to the decomposition \eqref{eq:decomq}. In particular, we claim that Proposition \ref{prop:redu} is a direct consequence of the following.

\begin{proposition}[Dynamics of equation \eqref{eq:q}] \label{prop:dya} For all $A \gg 1$, there exists $s_0 = s_0(A) \gg 1$ such that if $q(s) \in \Vc_A(s)$ for all $s \in [\tau, \tau_1]$ with $\tau_1 \geq \tau \geq s_0$, then the following estimates hold for all $s \in [\tau, \tau_1]$:
\begin{itemize}
\item[(i)] \textup{(Control of the finite dimensional part)} 
\begin{align}
&\left|q_k'(s) - \left(1 - \frac{k}{2m}\right)q_k\right| \leq \frac{C}{s^2} \quad \textup{for}\quad 0 \leq k \leq 2m-1.\label{eq:ODEqk}\\
& \quad  \left|q_{2m}'(s) + \frac{2}{s}q_{2m}\right| \leq \frac{CA^3}{s^3},\label{eq:ODEq2m}\\
& \qquad  |q_j(s)| \leq Ce^{-\left(\frac{j}{2m} - 1\right)(s-\tau)}|q_j(\tau)| + \frac{CA^{j-1}}{s^\frac{j+1}{2m}} \quad \textup{for}\quad 2m+1 \leq j \leq M.\label{eq:ODEqj}
\end{align}
\item[(ii)] \textup{(Control of the infinite dimensional and outer part)}
\begin{align}
&\left\|\frac{q_{_{M, \bot}}(y,s)}{1 + |y|^{M+1}}\right\|_{L^\infty(\Rb)} \leq Ce^{-\frac{M(s-\tau)}{2m}}\left\|\frac{q_{_{M, \bot}}(y,\tau)}{1 + |y|^{M+1}}\right\|_{L^\infty(\Rb)} + \frac{CA^M}{s^\frac{M+2}{2}}.\label{eq:contrq-M}\\
& \quad \left\|q_e(s)\right\|_{L^\infty(\Rb)} \leq C e^{-\frac{(s -\tau)}{2(p-1)}}\|q_e(\tau)\|_{L^\infty(\Rb)} + \frac{CA^{M+1}}{s^\frac{1}{2m}}(1  +s- \tau).\label{eq:contr:qe}
\end{align}
\end{itemize}
\end{proposition}
\begin{remark} Note that the sharp ODE \eqref{eq:ODEq2m} comes from the precise choice of the constant $B_{m,p}$ appearing in the profile $\Phi$ defined in \eqref{def:Phipro}. Other choices would give the error of size $\Oc\left(\frac{1}{s^2}\right)$ which is too large to close the estimate for $q_{2m}$. 
\end{remark}

\medskip

Let us postpone the proof of Proposition \ref{prop:dya} and proceed with the proof of Proposition \ref{prop:redu}.
\begin{proof}[Proof of Proposition \ref{prop:redu} assuming Proposition \ref{prop:dya}] Since the argument of the proof is similar to what was done  in \cite{MZdm97}, we only sketch the proof for the reader's convenience. We recall from the assumption that 
\begin{equation}\label{eq:qinVaassump}
q(s) \in \Vc_A(s) \quad \textup{for all} \quad s \in [s_0, s_1] \quad \textup{and}\quad q(s_1) \in \partial \Vc_A(s_1).
\end{equation}
Thus, part $(i)$ of Proposition \ref{prop:redu} will be proved if we show that all the bounds given in Definition \ref{def:VA} can be improved, except for the first $2m$ components $q_k$ with $0 \leq k \leq 2m-1$, in the sense that for all $s \in [s_0, s_1]$: 
\begin{align}
|q_{2m}(s)| < \frac{A^2 \log s}{s^2}, \quad |q_{k}(s)| < \frac{A^k}{s^\frac{k + 1}{2m}} \quad \textup{for}\quad 2m+1 \leq k \leq M, \label{eq:contrq2mqk}\\
\left\|\frac{q_{_{M, \bot}}(y,s)}{1 + |y|^{M+1}}\right\|_{L^\infty} < \frac{A^{M+1}}{s^\frac{M+2}{2m}}, \quad \|q_e(s)\|_{L^\infty} < \frac{A^{M+2}}{s^\frac{1}{2m}}.\label{eq:contrqbqe}
\end{align}
We first deal with $q_{2m}$.  We argue by contradiction by assuming that there is $s_* \in [s_0, s_1]$ such that for all $s \in [s_0, s_*)$,
$$|q_{2m}(s)| < \frac{A^2\log s}{s^2}, \quad |q_{2m}(s_*)| = \frac{A^2 \log s_*}{s_*^2},$$
(note that the existence of $s_*$ is guaranteed by item $(i)$ of Proposition \ref{prop:initdata}). Considering $q_{2m}(s_*) > 0$ and using minimality (the case $q_{2m}(s_*) < 0$ is similar), we have on the one hand, 
$$q'_{2m}(s_*) \geq A^2\frac{d}{ds}\left(\frac{\log s_*}{s_*^2}\right) = \frac{A^2}{s_*^3} - \frac{2A^2 \log s_*}{s_*^3},$$
which holds thanks to \eqref{eq:qinVaassump}. On the other hand, we obtain from the ODE \eqref{eq:ODEq2m} that 
$$q'_{2m}(s_*) \leq \frac{-2A^2\log s_*}{s_*^3} + \frac{CA}{s_*^3},$$
and the contradiction follows if $A \geq 2C + 1$.\\
As for $q_j$ with $2m+1 \leq j \leq M$ and $q_{_{M,\bot}}$ and $q_e$, we argue as follows: Let $\lambda = \log A$ and assume that  $s_0 \geq \lambda$ such that for all $\tau \geq s_0$ and $s \in [\tau, \tau + \lambda]$ we have 
$$\tau \leq s \leq \tau  + \lambda \leq \tau + s_0 \leq 2\tau,\quad  \textup{hence,}\quad \frac{1}{\tau} \sim \frac{1}{s}.$$
We distinguish into two cases:\\
- For $s - s_0 \leq \lambda$, we use estimates \eqref{eq:ODEqj}, \eqref{eq:contrq-M} and \eqref{eq:contr:qe} with $\tau = s_0$ together with Proposition \ref{prop:initdata} to find that 
\begin{align*}
&|q_j(s)| \leq \frac{C(1 + A^{j-1})}{s^\frac{j+1}{2m}} < \frac{A^j}{s^\frac{j+1}{2m}} \quad \textup{for} \quad 2m+1 \leq j \leq M,\\
& \left\|\frac{q_{_{M, \bot}}(y,s)}{1 + |y|^{M+1}} \right\|_{L^\infty} \leq \frac{C(1 + A^M)}{s^\frac{M+2}{2m}} < \frac{A^{M+1}}{s^\frac{M+2}{2m}},\\
& \|q_e(s)\|_{L^\infty} \leq \frac{C(1 + A^{M+1} \log A)}{s^\frac{1}{2m}} < \frac{A^{M+2}}{s^\frac{1}{2m}}, 
\end{align*} 
for $A$ large enough. Hence, estimates \eqref{eq:contrq2mqk} and \eqref{eq:contrqbqe} hold for $s -s_0 \leq \lambda$. \\
- For $s - s_0 > \lambda$, we use estimates \eqref{eq:ODEqj}, \eqref{eq:contrq-M} and \eqref{eq:contr:qe} with $\tau = s - \lambda > s_0$ together with \eqref{eq:qinVaassump} to write (remember that $s \geq \lambda/2$) 

\begin{align*}
&|q_j(s)| \leq e^{-\left(\frac{j}{2m} - 1\right) \lambda} \frac{A^j}{(s/2)^\frac{j  +1}{2m}} + \frac{CA^{j-1}}{s^\frac{j+1}{2m}} < \frac{A^j}{s^\frac{j+1}{2m}} \quad \textup{for}\;\; 2m + 1 \leq j \leq M,\\
& \left\|\frac{q_{_{M, \bot}}(y,s)}{1 + |y|^{M+1}} \right\|_{L^\infty} \leq Ce^{-\frac{M\lambda}{4m}}\frac{A^{M+1}}{(s/2)^\frac{M+2}{2m}} + \frac{CA^M}{s^\frac{M+2}{2}} < \frac{A^{M+1}}{s^\frac{M+2}{2m}},\\
& \|q_e(s)\|_{L^\infty} \leq Ce^{-\frac{\lambda}{2(p-1)}}\frac{A^{M+2}}{(s/2)^\frac{1}{2m}} + \frac{CA^{M+1}(1 + \log A)}{s^\frac{1}{2m}} < \frac{A^{M+2}}{s^\frac{1}{2m}}. 
\end{align*}
Therefore, estimates \eqref{eq:contrq2mqk} and \eqref{eq:contrqbqe} hold for all $s \in [s_0, s_1]$, hence, the conclusion of part $(i)$ of Proposition \ref{prop:redu} follows.\\

Part $(ii)$ of Proposition \ref{prop:redu} is a direct consequence of the dynamics of $q_k$ given in \eqref{eq:ODEqk} and \eqref{eq:qinVaassump}. Indeed, from part $(i)$ of Proposition \ref{prop:dya}, we know that $q_k(s_1) = \epsilon \frac{A}{s_1^2}$ for  some $k \in \{0, \cdots, 2m-1\}$ and $\epsilon = \pm 1$. Using estimate \eqref{eq:ODEqk} yields
$$\epsilon q'_k(s_1) \geq \left(1 - \frac{k}{2m}\right)\epsilon q_k(s_1) - \frac{C}{s_1^2} \geq \left( (1 - \frac{k}{2m})A - C\right)\frac{1}{s_1^2}.$$ 
Thus, for $0 \leq k\leq \frac{k}{2m}$ and $A$ large enough, we have $\epsilon q_k'(s_1) > 0$. Therefore, $q_k$ is traversal outgoing to the boundary curve $s \mapsto \epsilon As^{-2}$ at $s = s_1$. This concludes the proof of Proposition \ref{prop:redu}, assuming that Proposition \ref{prop:dya} holds.
\end{proof}

We now give the proof of Proposition \ref{prop:dya} to complete the proof of Proposition \ref{prop:redu}. We divide the proof into two subsections according to the two parts of Proposition \ref{prop:dya}.
\subsection{Control of the finite dimensional part.}
We prove item $(i)$ of Proposition \ref{prop:dya} in this part. A direct projection of equation \eqref{eq:q} on the eigenfunction $\psi_k$ for $0 \leq k \leq M$ yields 
\begin{equation}\label{eq:idqk}
q_k'(s) - \left(1 - \frac{k}{2m}\right)q_k(s) = \Pi_k\Big(Vq + B(q) + R\Big). 
\end{equation}
\paragraph{Estimate of $\Pi_k(Vq)$.} We claim the following.
\begin{lemma}[Expansion of $V$] \label{lemm:V} The potential $V$ defined by \eqref{def:V} satisfies the estimate
\begin{equation}\label{est:V1}
|V(y,s)| \leq \frac{C(1 + |y|^{2m})}{s}, \quad \forall y \in \Rb, \; s \geq 1,
\end{equation}
and admits the following uniform expansion for all $n \in \mathbb{N}^*$,
\begin{equation}\label{est:Vk}
V(y,s) = \sum_{j = 1}^n \frac{1}{s^j}V_{j}(y) + \Oc\left(\frac{1 + |y|^{2m(n+1)}}{s^{n+1}}\right), \quad \forall |y| \leq s^\frac{1}{2m},\; s > 1,
\end{equation}
where $V_j$'s are even polynomial of degree $2mj$. More precisely, we have 
\begin{equation}\label{def:V1}
V_1(y) = -\frac{2}{\bar \mu_{2m}}\psi_{2m} \quad \textup{with}\quad \bar \mu_{2m} = \big<\psi_{2m}^2, \psi_{2m}^*\big>.
\end{equation}
\end{lemma}
\begin{proof} Estimate \eqref{est:V1} is trivial.  Estimate \eqref{est:Vk} follows from a Taylor expansion in the variable $z = \frac{|y|^{2m}}{s}$. Formula \eqref{def:V1} comes from the definitions \eqref{def:appPro} and \eqref{def:psibeta} of $\varphi$ and $\psi_{2m}$.
\end{proof}
From Lemma \ref{lemm:V}, we obtain the following estimate for $\Pi_k(Vq)$.
\begin{lemma}[Estimate of $\Pi_k(Vq)$] \label{lemm:projVq} Under the assumption of Proposition \ref{prop:dya}, we have 
\begin{align*}
&\Big|\Pi_k(Vq)\Big| \leq \frac{C}{s^2} \quad \textup{for}\quad 0 \leq k \leq 2m-1,\\
& \quad \left|\Pi_{2m}(Vq) + \frac{2}{s}q_{2m}\right| \leq \frac{C A}{s^3},\\
& \qquad \Big|\Pi_k(Vq)\Big| \leq \frac{CA^{k-2}}{s^\frac{k+1}{2m}} \quad \textup{for}\quad 2m + 1 \leq k \leq M.
\end{align*}
\end{lemma}
\begin{proof}  By Definition \ref{def:VA}, we have 
\begin{equation}\label{est:qVA}
|q(y,s)| \leq \frac{CA^{M+1}}{s^{1 + \frac{1}{m}}}(1 + |y|^{M+1}) \quad \textup{for all} \;\; y \in \Rb.
\end{equation}
Using this, the estimate \eqref{est:V1} and noting from the definition \eqref{def:psibeta*} and \eqref{est:F} that $\psi_k^*$ is exponentially decaying, we obtain for $0 \leq k \leq 2m-1$,
\begin{align*}
\left|\Pi_k(Vq)\right| \leq \frac{CA^{M+1}}{s^{2 + \frac{1}{m}}}\int_{\Rb} (1 + |y|^{M+1+2m})\psi_k^*(y)dy \leq \frac{C}{s^2}. 
\end{align*}
For $2m + 1 \leq k \leq M$, we write from the decomposition \eqref{eq:decomq},
\begin{align*}
\Pi_k(Vq) = \sum_{i = 0}^{M}q_i(s)\int_{\Rb} V(y,s)\psi_i(y) \psi_k^*(y) dy + \int_{\Rb}V(y,s) q_{_{M, \bot}}(y,s) \psi_k^*(y)dy.
\end{align*}
Using \eqref{est:V1} and the  bound of $q_{_{M, \bot}}$ given in Definition \ref{def:VA} yields 
$$\left|\int_{\Rb}V(y,s) q_{_{M, \bot}}(y,s) \psi_k^*(y)dy\right| \leq  \frac{CA^{M+1}}{s^{1 + \frac{M+1}{2m}}} \int_{\Rb}(1 + |y|^{M + 1 + 2m}) \psi_k^*(y)dy \leq \frac{CA^{M+1}}{s^{1 + \frac{M+1}{2m}}}.$$
From the orthogonality \eqref{eq:orths}, we note that $\big<P_n, \psi_k^*\big> = 0$ for all polynomial $P_n$ of degree $n \leq k - 1$. We then use \eqref{est:Vk} and \eqref{est:qVA} to estimate 
\begin{align*}
\sum_{i = 0}^{M}q_i(s)\int_{\Rb} V(y,s)\psi_i(y) \psi_k^*(y) dy &= \sum_{i = 0}^M\sum_{j = 1}^n \frac{q_i(s)}{s^j}\int_{|y| \leq s^\frac{1}{2m}} V_j(y)\psi_i(y)\psi_k^*(y)dy \\
& \quad + \Oc\left(\frac{A^{M+1}}{s^{n + 1 + \frac{1}{m}}} \int_{|y| \leq s^\frac{1}{2m}}(1 + |y|^{2m(n+1) + M +1})|\psi^*_k(y)|dy \right)\\
& \qquad + \Oc\left(\frac{A^{M+1}}{s^{2 + \frac{1}{m}}} \int_{|y| \geq s^\frac{1}{2m}}(1 + |y|^{2m + M + 1})|\psi^*_k(y)|dy \right).
\end{align*}
The last term is bounded by $\Oc(e^{-cs})$ because of the exponential decay of $\psi_k^*$. By taking $n \in \mathbb{N}^*$ such that $n + 1 + \frac{1}{m} \geq \frac{k + 2}{2m}$, the second term is bounded by $\Oc\left(A^{M+1}s^{- \frac{k + 2}{2m}}\right)$. Using the fact that $V_j \psi_i$ is polynomial of degree $2mj + i$, we see that $\big<V_j \psi_i, \psi_k^*\big> = 0$ for $2mj + i \leq k - 1$. Combining this with the bounds of $q_i$ given in Definition \ref{def:VA} of $\Vc_A$, we obtain the rough estimate 
\begin{align*}
&\left|\sum_{i = 0}^M\sum_{j = 1}^n \frac{q_i(s)}{s^j}\int_{|y| \leq s^\frac{1}{2m}} V_j(y)\psi_i(y)\psi_k^*(y)dy\right| \leq \sum_{i = 0}^M \sum_{j = 1, 2jm + i \geq k}^n \frac{CA^i}{s^{j + \frac{i + 1}{2m}}}\\
& \quad \leq \frac{CA^{k-2}}{s^\frac{k+1}{2m}} + \sum_{i = k-1}^M \frac{CA^M}{s^{1 + \frac{i + 1}{2m}}} \leq \frac{CA^{k-2}}{s^\frac{k+1}{2m}}.
\end{align*}
As for $k = 2m$, we need to use the precise definition \eqref{def:V1} of $V_1$ and process similarly as for $k \geq 2m+1$. Indeed, from decomposition \eqref{eq:decomq} and expansion \eqref{est:Vk} with $n = 2$, we have
\begin{align*}
\Pi_{2m}(Vq) - \frac{q_{2m}}{s}\int_{|y| \leq s^\frac{1}{2m}}V_1(y)\psi_{2m}\psi_{2m}^* dy  = \sum_{i = 0, i \ne 2m}^M \frac{q_i(s)}{s} \int_{|y| \leq s^\frac{1}{2m}} V_1 \psi_i \psi_{2m}^* + \Oc\left(\frac{A^{M+1}}{s^{3 + \frac{1}{m}}}\right).
\end{align*}
Using the definition \eqref{def:V1} of $V_1$, the bound of $q_i$ given in Definition \ref{def:VA}  and the fact coming from the definition \eqref{def:psibeta} of $\psi_\beta$ and the orthogonality \eqref{eq:orths} that 
$$\int_{\Rb}\psi_{2m} \psi_i \psi_{2m}^*dy = \int_{\Rb} \big(c_0\psi_{2m + i} + c_1\psi_i + c_2 \psi_{i - 2m}\big)\psi_{2m}^* = 0 \quad \textup{for}\;\; 2m +1  \leq i \leq 4m - 1,$$
we arrive at
$$\left|\Pi_{2m}(Vq) + \frac{2}{s}q_{2m}\right| \leq \sum_{i = 0}^{2m-1} \frac{CA}{s^3} + \sum_{4m}^M \frac{CA^M}{s^{1 + \frac{i + 1}{2m}}} + \frac{CA^{M+1}}{s^{3 + \frac{1}{m}}} \leq \frac{CA}{s^3}.$$
This concludes the proof of Lemma \ref{lemm:projVq}.
\end{proof}

We now turn to the estimate of the main contribution of the nonlinear term  $B(q)$ under the projection $\Pi_k$. We begin with the following expansion. 
\begin{lemma}[Expansion of $B(q)$] \label{lemm:exBq} For all $|q| < 1$, $s \geq s_0 \gg 1$ and $|y| \leq s^\frac{1}{2m}$, the function $B(q)$ defined by \eqref{def:B} admits the uniform expansion
\begin{equation}\label{eq:exBq}
\left|B(q) - \sum_{j = 2}^{M+1} q^j \sum_{l = 0}^M \frac{1}{s^l} \left(B_{j,l}(ys^{-\frac{1}{2m}})  + \tilde{B}_{j,l}(y,s)\right)\right| \leq C|q|^{M+2} + \frac{C}{s^{M+1}},
\end{equation}
where $B_{j,l}(z)$'s are even polynomials of degree $l$ and 
$$|\tilde{B}_{j,l}(y,s)| \leq \frac{C(1 + |y|^{M+1})}{s^\frac{M+1}{2m}}.$$
Furthermore, we have the estimate for all $y \in \Rb$ and $s \geq 1$,
\begin{equation}\label{eq:Bqbound}
|B(q)| \leq C|q|^{\bar p} \quad \textup{with} \;\; \bar p = \min\{2,p\}.
\end{equation}
\end{lemma}
 \begin{proof} We first note from the definition \eqref{def:appPro} of $\varphi$ that there exist two positive constants $c_0, C_0$ such that $c_0 \leq \varphi(y,s) \leq C_0$ for $|y| \leq s^\frac{1}{2m}$. Thus, a Taylor expansion of $B(q, \varphi)$ in terms of $q$ yields
$$\left|B(q) - \sum_{j = 2}^{M+1}B_j(\varphi) q^j\right| \leq C|q|^{M+1}, $$
where $B_{j}(\varphi)$ admits the following expansion in terms of $\frac{1}{s}$, 
$$\left|B_j(\varphi) - \sum_{l = 0}^M \frac{1}{s^l}B_{j,l}(\Phi)\right| \leq \frac{C}{s^{M+1}}.$$
Now, a Taylor expansion of $B_{j,l}(\Phi)$ in terms of the variable $z = ys^{-\frac{1}{2m}}$ yields the desired result. This concludes the proof of Lemma \ref{lemm:exBq}.
\end{proof}
With Lemma \ref{lemm:exBq} at hand, we estimate the main contribution of $B(q)$ under the projection $\Pi_k$. We claim the following.
\begin{lemma}[Estimate for $\Pi_k\big(B(q)\big)$] \label{lemm:projBq} Under the assumption of Proposition \ref{prop:dya}, we have 
\begin{align*}
&\left|\Pi_k (B(q))\right| \leq \frac{C}{s^2} \quad \textup{for} \;\; 0\leq k \leq 2m - 1, \quad \left|\Pi_{2m} (B(q))\right| \leq \frac{C}{s^3},\\
& \quad \left|\Pi_k(B(q))\right| \leq \frac{CA^k}{s^{\frac{k+2}{2m}}} \quad \textup{for} \;\; 2m+1\leq k \leq M.
\end{align*}
\end{lemma}
\begin{proof} From \eqref{eq:exBq}, the exponential decay of $\psi_k^*$ and \eqref{est:qVA}, we have
\begin{align*}
\Pi_k(B(q)) = \sum_{j = 2}^{M+1}\sum_{l = 0}^M \frac{1}{s^l} \int_{|y| \leq s^\frac{1}{2m}}q^j B_{j,l} \psi_k^* dy + \Oc\left(\frac{A^{M+1}}{s^{2(1 + 1/m) + \frac{M+1}{2m}}}\right) + \Oc(e^{-cs}).
\end{align*}
From the decomposition \eqref{eq:decomq}, we write 
$$q^j = \big(\sum_{i = 0}^M q_i \psi_i + q_{_{M,\bot}}\big)^j \equiv \big(q_< + q_{_{M,\bot}}\big)^j \quad \textup{for} \quad j \geq 2.$$
From Definition \ref{def:VA} of $\Vc_A$ and the uniform boundedness of $q$, we obtain the bound for $j \geq 2$, 
$$|q^j - q_<^j| \leq C\Big(|q_{_{M,\bot}}|^j + |q_<|^{j-1}|q_{_{M,\bot}}|\Big) \leq \frac{CA^{j(M+1)}}{s^\frac{M+4 + 2m}{2m}}(1 + |y|^{j(M+1)}).$$
Therefore, 
\begin{align*}
\Pi_k(B(q)) = \sum_{j = 2}^{M+1}\sum_{l = 0}^M \frac{1}{s^l} \int_{|y| \leq s^\frac{1}{2m}}q_<^j B_{j,l} \psi_k^* dy +  \Oc\left(\frac{A^{j(M+1)}}{s^{\frac{M+4 + 2m}{2m}}}\right).
\end{align*}
We only handle the case $k = 2m$, which is the most delicate. The case $k \ne 2m$ can be processed similarly and we omit it. From \eqref{est:qVA}, we have
$$\sum_{j = 3}^{M}\sum_{l = 0}^M \frac{1}{s^l} \int_{|y| \leq s^\frac{1}{2m}}\big|q_<^j B_{j,l} \psi_{2m}^*\big| dy + \sum_{l = 1}^M \frac{1}{s^l} \int_{|y| \leq s^\frac{1}{2m}}\big|q_<^2 B_{j,l} \psi_{2m}^*\big| dy \leq \frac{CA^{j(M+1)}}{s^{3 + \frac{1}{m}}} \leq \frac{C}{s^3}.$$
It remains to control the term $I:=\int_{|y| \leq s^\frac{1}{2m}}q_<^2 \psi_{2m}^*dy$. By the definition of $q_<$, we expand 
\begin{align*}
q_<^2 &= \left(\sum_{i = 0}^{2m} q_i \psi_i\right)^2 + 2\left(\sum_{i=0}^{2m}q_i \psi_i\right) \left(\sum_{i' = 2m + 1}^M q_{i'} \psi_{i'}\right) + \left(\sum_{i' = 2m+1}^{M} q_{i'} \psi_{i'}\right)^2\\
&= I_1 + I_2 + I_3.
\end{align*}
From Definition \ref{def:VA}, we have $|I_1| \leq \frac{CA^4 \log^2}{s^4} (1 + |y|^{4m})$ and $|I_2| \leq \frac{CA^{M+4} \log s}{s^{3 + 1/m}}(1 + |y|^{M + 2m})$. Hence, 
$$\int_{|y| \leq s^\frac{1}{2m}}\big|(I_1 + I_2) \psi_{2m}^* dy \big| \leq \frac{CA^4 \log^2}{s^4} + \frac{CA^{M+4} \log s}{s^{3 + 1/m}} \leq \frac{1}{s^3}.$$ 
Note from the orthogonality relation \eqref{eq:orths} that if $i'+ l' \ne J2m$ for $J \in \mathbb{N}^*$, we have
\begin{align*}
\big<\psi_{i'} \psi_{l'}, \psi_{2m}^* \big> & = \left< \left(\sum_{i = 0}^{\lfloor \frac{i'}{2m}\rfloor} c_{i', i}\partial_y^{2im} y^{i'}\right) \left( \sum_{l = 0}^{\lfloor \frac{l'}{2m}\rfloor} c_{l', l}\partial_y^{2lm} y^{l'}\right), \psi_{2m}^* \right> \\
& = \left< \sum_{j = 0}^{\lfloor \frac{i'}{2m}\rfloor + \lfloor \frac{l'}{2m}\rfloor} \hat c_{i'+l',j} y^{i' + l' - 2jm}, \psi_{2m}^*\right> = \left< \sum_{j = 0}^{\lfloor \frac{i'}{2m}\rfloor + \lfloor \frac{l'}{2m}\rfloor} \tilde c_{i'+l',j} \psi_{i' + l' - 2jm}, \psi_{2m}^*\right> = 0.
\end{align*}
Therefore, we estimate 
\begin{align*}
\int_{|y| \leq s^\frac{1}{2m}}I_3 \psi_{2m}^* dy &= \sum_{i', l' = 2m+1}^M c_{i',l'}q_{i'}(s) q_{l'}(s) \int_{\Rb} \psi_{i'}\psi_{l'} \psi_{2m}^* dy + \Oc(e^{-cs})\\
& \quad = \sum_{i', l' = 2m+1, i'+l' = J2m}^M c_{i',l'}q_{i'}(s) q_{l'}(s) \int_{\Rb} \psi_{i'}\psi_{l'} \psi_{2m}^* dy + \Oc(e^{-cs}).
\end{align*}
for $J \in \mathbb{N}^*$. Since $i' + l' \geq 4m + 2$, it follows that $J \geq 3$. From Definition \ref{def:VA}, we have the bound $|q_{i'} q_{l'}| \leq \frac{A^{i' + l'}}{s^{\frac{i' + l' + 2}{2m}}} \leq \frac{A^{i' + l'}}{s^{3 + \frac{1}{m}}} \leq \frac{1}{s^3}$. Therefore, $\left|\int_{|y| \leq s^\frac{1}{2m}}I_3 \psi_{2m}^* dy\right| \leq \frac{1}{s^3}$. This concludes the proof of Lemma \ref{lemm:projBq}.
\end{proof}

We now deal with the generated error term $R$. We begin with the following expansion. 
\begin{lemma}[Expansion of $R$] \label{lemm:exR} The function $R$ defined by \eqref{def:R} satisfies 
$$\|R(s)\|_{L^\infty} \leq \frac{C}{s},$$
and admits the following expansion: for all $n \in \mathbb{N}^*$, $|y| \leq s^\frac{1}{2m}$ and $s \geq 1$:
\begin{equation}\label{eq:exR}
\left|R(y,s) - \sum_{j = 1}^{n-1} \frac{1}{s^{j+1}} R_j(y)\right| \leq \frac{C(1 + |y|^{2mn})}{s^{n + 1}},
\end{equation}
where $R_j$'s are polynomials of degree $2mj$ of the form $R_j(y) = \sum_{i = 0}^j d_i y^{2m i}$. Moreover, the coefficient of degree $2m$ of $R_1$ is identically zero, hence, $\big<R_1, \psi^*_{2m} \big> = 0$.
\end{lemma}
\begin{proof} Let $z = y s^{-\frac{1}{2m}}$ and note that $\Phi(z)$ satisfies equation \eqref{eq:Phi}. Therefore, we rewrite \eqref{def:R} as follows:
\begin{equation} \label{eq:R1}
R(y,s) = \frac{z}{2m s} \cdot \nabla \Phi + \frac{A_{m,p}}{s^2} + \frac{1}{s}\Am \Phi - \frac{A_{m,p}}{(p-1)s} + Q\left(\Phi + \frac{A_{m,p}}{s}\right) - Q(\Phi),
\end{equation}
where $Q(h) = h^p$. Since $c_0 \leq \Phi(z) \leq C_0$ for all $|z| \leq 1$ with $c_0, C_0$ some positive constants, we have the following expansion of $Q$,
$$\left|Q\left(\Phi + \frac{A_{m,p}}{s}\right) - Q(\Phi) - \sum_{j = 1}^J\frac{1}{s^j}Q_j(\Phi) \right| \leq \frac{C}{s^{J+1}} \quad \textup{for} \;\;J \in \mathbb{N}^*.$$
Then, we expand $Q_j$ and all the remaining terms in \eqref{eq:R1} in power series of $Z = z^{2m}$ to obtain the desired result. Note that the coefficient of $\frac{1}{s}$ in the expansion of $R$ (after an elementary computation) is given by 
$$A_{m,p} - \frac{(2m)! \kappa B_{m,p}}{p-1} = 0,$$
where we used \eqref{def:Bpm}. Moreover, $R_1(y) = C_1 y^{2m} + C_2$, where $C_2 = C_2(p,m)$ and  
$$C_1 = \frac{\kappa B_{m,p}}{p-1}\left(\frac{(4m)!}{(2m)!} \frac{p B_{m,p}}{2(p-1)} - \frac{pA_{m,p}}{\kappa} - 1\right) = 0.$$
Again, the precise values of $B_{m,p}$ and $A_{m,p}$ given in \eqref{def:Bpm} are crucial in deriving that $C_1$ is identically zero. Therefore, the orthogonality relation \eqref{eq:orths} yields $\big<R_1, \psi_{2m}^*\big> = 0$. This concludes the proof of Lemma \ref{lemm:exR}.
\end{proof}
As a direct consequence of Lemma \ref{lemm:exR}, we have the following.
\begin{lemma}[Estimate of $\Pi_k(R)$] \label{lemm:projR} Under the assumption of Proposition \ref{prop:dya}, we have 
\begin{align*}
&\Pi_k(R) = \Oc(s^{-M}) \quad \textup{for} \quad (k\; \textup{mod} \; 2m) \ne 0,\\
&\big|\Pi_0(R)\big| \leq \frac{C}{s^2}, \quad  \big|\Pi_{2m}(R)\big| \leq \frac{C}{s^3},\quad  \big|\Pi_{2mi}(R)\big| \leq \frac{C}{s^{i + 1}} \quad \textup{for} \;\; i \in \mathbb{N}^*.
\end{align*}
\end{lemma}
\begin{proof} The proof directly follows from the expansion \eqref{eq:exR} and the fact that 
$$\big<y^{2mj}, \psi_k^*\big> = \sum_{i = 0}^j a_i\big< \psi_{2mi}, \psi_k^*\big> = 0 \quad \textup{for}\quad (k\; \textup{mod} \; 2m) \ne 0.$$
For $k \in \mathbb{N}$ with $(k\; \textup{mod} \; 2m) \ne 0$, we use the expansion \eqref{eq:exR} with $n = M-1$ (we can replace $M$ by any positive integer $L \gg 1$) and write 
\begin{align*}
|\Pi_k(R)| &= \left|\int_{\Rb} R(y,s)\psi_k^*(y)dy \right| = \left|\int_{|y| \leq s^{\frac{1}{2m}}} R(y,s)\psi_k^*(y)dy + \int_{|y| \geq s^{\frac{1}{2m}}} R(y,s)\psi_k^*(y)dy \right|\\
& \leq \sum_{j = 1}^{M-2} \int_{|y| \geq s^{\frac{1}{2m}}} \frac{|R_j(y)|}{s^{j+1}} |\psi_k^*(y)| dy + \frac{C}{s^M}\int_{\Rb}(1 + |y|^{2m(M-1)})|\psi_k^*(y)|dy\\
& \quad \qquad  + \frac{C}{s} \int_{|y| \geq s^\frac{1}{2m}}|\psi_k^*(y)|dy \leq \frac{C}{s^M} + Ce^{-cs} \leq \frac{C}{s^M}.
\end{align*}
For $(k\; \textup{mod} \; 2m) \ne 0$, i.e. $k = 2m i$ for some $i \in \mathbb{N}$, we use \eqref{eq:exR} with $n = i + 1$ to get the conclusion. This ends the proof of Lemma \ref{lemm:projR}. 
\end{proof}

\medskip

\noindent A collection of all estimates given in Lemmas \ref{lemm:projVq}, \ref{lemm:projBq}, \ref{lemm:projR} and equation \eqref{eq:idqk} yields the conclusion of part $(i)$ of Proposition \ref{prop:dya}.

\subsection{Control of the infinite dimensional and the outer part.} We prove item $(ii)$ of Proposition in this part. We first deal with the infinite dimensional part $q_{_{M,\bot}}$, then the outer part $q_e$.\\

\paragraph{\underline{Control of $q_{_{M,\bot}}$}:} Applying $\Pi_{_{M, \bot}}$ to equation \eqref{eq:q}  and using the fact that $\Pi_{_{M, \bot}}\psi_n = 0$ for all $n \leq M$ (see \eqref{def:PiMbot}), we obtain 
\begin{equation}\label{eq:qM-}
\partial_s q_{_{M, \bot}} = \Ls_m q_{_{M, \bot}} + \Pi_{_{M, \bot}} \left(Vq + B(q) + R\right).
\end{equation}
By definition, we have $\Pi_{_{M, \bot}}(y^k) = 0$ for all $k \leq M$. Using the uniform bound $\|R(s)\|_{L^\infty} \leq \frac{C}{s}$ and the expansion \eqref{eq:exR} with $n = \frac{M}{2m}$ (see \eqref{def:M} for the definition of $M$) and noting that $\Pi_{_{M, \bot}}(R_j) = 0$ for $j \leq \frac{M}{2m}$, we obtain
\begin{align}
\left|\Pi_{_{M, \bot}}(R)\right| &\leq \frac{C(1 + |y|^{M + 2m})}{s^{\frac{M}{2m} + 2}}\mathbf{1}_{|y| \leq s^\frac{1}{2m}} + |R(y,s)|\mathbf{1}_{|y| \geq s^\frac{1}{2m}} \leq  \frac{C(1 + |y|^{M + 1})}{s^{\frac{M + 1}{2m} + 1}}.\label{eq:estPiMR}
\end{align}
As for the estimate of $\Pi_{_{M, \bot}}(Vq)$, we claim the following. 
\begin{lemma}[Estimate of $\Pi_{_{M, \bot}}(Vq)$] \label{lemm:PiMVq} Under the assumption of Proposition \ref{prop:dya}, we have 
\begin{equation}\label{eq:estPiMVq}
\left\|\frac{\Pi_{_{M, \bot}}(Vq)}{1 + |y|^{M+1}}\right\|_{L^\infty(\Rb)} \leq \left(\|V(s)\|_{L^\infty(\Rb)} + \frac{C}{s}\right)\left\|\frac{q_{_{M, \bot}}}{1+|y|^{M+1}} \right\|_{L^\infty(\Rb)} + \frac{CA^M}{s^{\frac{M + 2}{2m}}}.
\end{equation}
\end{lemma}
\begin{proof}  By definition, we write 
$$\Pi_{_{M, \bot}}(Vq) = Vq_{_{M, \bot}} - \Pi_{_{M, <}}\big(Vq_{_{M, \bot}}\big) + \Pi_{_{M, \bot}}\big( V q_{_{M,<}}\big),$$
where $q_{_{M,<}} = \Pi_{_{M,<}}(q) = \big(\textup{Id} - \Pi_{_{M,\bot}}\big)(q)$. Using estimate \eqref{est:V1}, we derive 
$$\left\|\frac{Vq_{_{M, \bot}}}{1 + |y|^{M+1}}\right\|_{L^\infty} + \left\|\frac{\Pi_{_{M, <}}\big(Vq_{_{M, \bot}}\big)}{1 + |y|^{M+1}}\right\|_{L^\infty} \leq \left(\|V(s)\|_{L^\infty(\Rb)} + \frac{1}{s} \right)\left\|\frac{q_{_{M, \bot}}}{1 + |y|^{M+1}}\right\|_{L^\infty}.$$
As for the control of $\Pi_{_{M, \bot}}\big( V q_{_{M,<}}) = \sum_{i \leq M} \Pi_{_{M, \bot}}\big(q_iV\psi_i) $, we argue as follows:\\
- If $M - i = 0\; \textup{mod}\; 2m$, we take $n = \frac{M-i}{2m}$ in the expansion \eqref{est:Vk} and write 
\begin{align*}
\Pi_{_{M, \bot}}\big( V q_i \psi_i\big) = \sum_{j = 1}^{n} \frac{q_i}{s^j}\Pi_{_{M, \bot}}\big(V_j\psi_i\big) + \Pi_{_{M, \bot}}\Big(\tilde{V}_{n}q_i \psi_i\Big),
\end{align*} 
where $|\tilde{V}_{n}(y,s)| \leq \frac{C(1 + |y|^{2m(n + 1)})}{s^{n + 1}}$. Using the fact that $\Pi_{_{M, \bot}}(y^n) = 0$ for $n \leq M$, we have $\sum_{j = 1}^n\Pi_{_{M,\bot}}(V_jq_i) = 0$. Hence, 
$$\left|\Pi_{_{M, \bot}}\big( V q_i \psi_i\big)\right| \leq \frac{C |q_i|}{s^{\frac{M - i}{2m} + 1}} \leq \frac{CA^M}{s^{\frac{M + 2}{2m}}}.$$
- If $M - i = l\; \textup{mod}\; 2m$ for some $l \in (1, 2m)$, we take $n = \frac{M - i - l}{2m}$ in the expansion \eqref{est:Vk} to deduce that 
$$\left|\Pi_{_{M, \bot}}\big( V q_i \psi_i\big)\right| \leq \frac{C |q_i|}{s^{\frac{M - i - l}{2m} + 1}} \leq \frac{CA^M}{s^{\frac{M + 2}{2m}}}.$$
This concludes the proof of Lemma \ref{lemm:PiMVq}.
\end{proof}

Concerning the control of $\Pi_{_{M, \bot}}(B(q))$, we claim the following.
\begin{lemma}[Estimate of $\Pi_{_{M, \bot}}(B(q))$] \label{lemm:PiMBq} Under the assumption of Proposition \ref{prop:dya}, we have 
\begin{equation}\label{eq:estPiMBq}
\left\|\frac{\Pi_{_{M, \bot}}(B(q))}{1 + |y|^{M+1}}\right\|_{L^\infty(\Rb)} \leq \frac{CA^{(M+2)^2}}{s^{\frac{M + 1 + \bar p}{2m}}} \quad \textup{with} \; \bar p = \min\{2,p\}.
\end{equation}
\end{lemma}
\begin{proof} Let $B_M$ be defined by 
\begin{equation}\label{def:BM}
B_M(y,s) = \Pi_{_{M, <}} \left[\sum_{j = 2}^{M+1}q^j\sum_{l = 0}^{M} \frac{1}{s^l}B_{j,l}(ys^{-\frac{1}{2m}}) \right],
\end{equation}
where $\Pi_{_{M,<}} = \textup{Id} - \Pi_{_{M, \bot}}$ and $B_{j,l}$'s are introduced in Lemma \ref{lemm:exBq}. Then we claim the following: for all $y \in \Rb$ and $s \geq 1$,
\begin{equation}\label{eq:claimBBM}
\left|B(q) - B_M\right| \leq \frac{CA^{(M+2)^2}}{s^{\frac{M + 1 + \bar p}{2m}}}(1 + |y|^{M+1}).
\end{equation} 
The estimate \eqref{eq:estPiMBq} simply follows from \eqref{eq:claimBBM} since $\Pi_{_{M, \bot}}(B_M) = 0$. Let us prove \eqref{eq:claimBBM}. In the region $|y| \geq s^\frac{1}{2m}$, we use \eqref{eq:Bqbound} and the bounds given in Definition \ref{def:VA} to estimate 
$$|B(q)| \leq \frac{CA^{\bar{p}(M+2)}}{s^\frac{\bar p}{2m}} \leq \frac{CA^{\bar p(M+2)}}{s^\frac{M + 1 + \bar p}{2m}} (1 + |y|^{M+1}) \quad \textup{for} \;\; |y| \geq s^\frac{1}{2m}.$$
In the region $|y| \leq s^\frac{1}{2m}$, we write from Lemma \ref{lemm:exBq} the expansion
\begin{align*}
&\left|B(q) - \Pi_{_{M,<}}\left(\sum_{j = 2}^{M+1}q^j\sum_{l = 0}^{M} \frac{1}{s^l}B_{j,l}(ys^{-\frac{1}{2m}})  \right)\right|\\
& \quad \leq \left|\Pi_{_{M, \bot}}\left(\sum_{j = 2}^{M+1}q^j\sum_{l = 0}^{M} \frac{1}{s^l}B_{j,l}(ys^{-\frac{1}{2m}})\right)\right| + \left|\sum_{j = 2}^{M+1}q^j\sum_{l = 0}^{M} \frac{1}{s^l} \tilde{B}_{j,l}(y,s) \right| + |q|^{M+2} + \frac{C}{s^{M+1}}.
\end{align*}
A similar argument as in the proof of Lemma \ref{lemm:projBq} shows that the coefficient of degree $k \geq M+1$ of the polynomial $\sum_{j = 0}^{M+1}q^j\sum_{l = 0}^{M} \frac{1}{s^l}B_{j,l}(ys^{-\frac{1}{2m}})$ is bounded by $\frac{A^k}{s^\frac{k + 2}{2m}}$, hence, 
$$\left|\Pi_{_{M, \bot}}\left(\sum_{j = 2}^{M+1}q^j\sum_{l = 0}^{M} \frac{1}{s^l}B_{j,l}(ys^{-\frac{1}{2m}})\right)\right| \leq \frac{CA^{2(M+1)}}{s^{\frac{M+3}{2m}}}(1 + |y|^{M+1}), \quad \forall |y| \leq s^\frac{1}{2m}.$$
From \eqref{est:qVA} and Remark \ref{remark:1}, we obtain 
$$\left|\sum_{j = 2}^{M+1}q^j\sum_{l = 0}^{M} \frac{1}{s^l} \tilde{B}_{j,l}(y,s) \right| \leq  \frac{C A^{(M+2)(M+1)}}{s^{ \frac {1}{m} + \frac{M+1}{2m}}}(1 + |y|^{M+1}) \leq \frac{CA^{(M+2)^2}}{s^{\frac{M + 1 + \bar p}{2m}}}(1 + |y|^{M+1}), $$
and 
$$|q|^{M+2} \leq C\left(\frac{A^{M+2}}{s^\frac{1}{2m}}\right)^{M + 1} \frac{A^{M+2}}{s^{1 + \frac{1}{m}}}(1 + |y|^{m+1})\leq \frac{CA^{(M+2)^2}}{s^{\frac{M + 1 + \bar p}{2m}}}(1 + |y|^{M+1}).$$ 
This completes the proof of \eqref{eq:claimBBM} as well as the proof of Lemma \ref{lemm:PiMBq}.
\end{proof}

\medskip

Plugging \eqref{eq:estPiMR}, \eqref{eq:estPiMVq} and \eqref{eq:estPiMBq} into equation \eqref{eq:qM-} yields 
$$\partial_s q_{_{M, \bot}} = \Ls_m q_{_{M, \bot}} + G_{_{M,\bot}},$$
where $G_{_{M, \bot}}$ satisfies for $s$ large enough,
$$\left\| \frac{G_{_{M, \bot}}(y,s)}{1 + |y|^{M+1}}\right\|_{L^\infty(\Rb)} \leq \|V(s)\|_{L^\infty(\Rb)}\left\|\frac{q_{_{M, \bot}}(y,s)}{1 + |y|^{M+1}} \right\|_{L^\infty(\Rb)} + \frac{CA^M}{s^{\frac{M + 2}{2m}}}.$$
Using the semigroup representation by $\Ls_m$, we write for all $s \in [\tau, \tau_1]$,
$$q_{_{M, \bot}}(s) = e^{(s-\tau)\Ls_m}q_{_{M, \bot}}(\tau) + \int_\tau^se^{(s - s')\Ls_m} G_{_{M, \bot}}(s') ds',$$
where $e^{s\Ls_m}$ is defined in Lemma \ref{lemm:semigLm}. Letting $\lambda(s) = \left\|\frac{q_{_{M, \bot}}(s)}{1 + |y|^{M+1}}\right\|_{L^\infty(\Rb)}$ and using $(iii)$ of Lemma \eqref{lemm:semigLm}, we have 
\begin{align*}
 \lambda(s) &\leq e^{-\frac{M}{2m}(s - \tau)}\lambda(\tau) + \int_{\tau}^s e^{-\frac{M}{2m}(s - s')}\left\|\frac{G_{_{M, \bot}}(s')}{1 + |y|^{M+1}}\right\|_{L^\infty(\Rb)}ds'\\
& \quad \leq e^{-\frac{M}{2m}(s - \tau)}\lambda(\tau) + \int_{\tau}^s e^{-\frac{M}{2m}(s - s')}\left(\|V(s')\|_{L^\infty(\Rb)}\lambda(s') + \frac{CA^M}{s'^{\frac{M+2}{2m}}} \right)ds'.
\end{align*}
Since we have fixed $M$ large such that $\|V\|_{L^\infty_{y,s}} \leq \frac{M}{4m}$ (see \eqref{def:M}), a standard Gronwall's argument applied to the function $e^\frac{Ms}{2m}\lambda(s)$ yields
$$e^{\frac{Ms}{2m}} \lambda(s) \leq e^{\frac{M(s -\tau)}{4m}}e^{\frac{M\tau}{2m}}\lambda(\tau) + Ce^{\frac{Ms}{2m}} \frac{A^M}{s^\frac{M+2}{2}},$$
from which we conclude the proof of \eqref{eq:contrq-M}.\\

\paragraph{\underline{Control of $q_e$}.} We write from \eqref{eq:q} the equation satisfied by $q_e = (1 - \chi(y,s))q$,
\begin{align*}
\partial_s q_e &= \left(\Ls_m  - \frac{p}{p-1}\right)q_e + (1 - \chi)(Q + R) + E_1 + E_2,
\end{align*}
where $R$ is defined by \eqref{def:R},
\begin{align*}
Q &= |q + \varphi|^{p-1}(q + \varphi) - \varphi^p,\quad E_1 = - q \big(\partial_s \chi + \Am\chi + \frac{1}{2m}y\cdot \nabla \chi\big),\\
E_2 &= \Am(\chi q) + q\Am \chi - \chi \Am q = \sum_{j = 1}^{2m - 1}c_j \nabla^j \big(q \nabla^{2m - j} \chi\big) .
\end{align*}
Using the semigroup representation by $\Ls_m$ and items $(i)-(ii)$ of Lemma \ref{lemm:semigLm}, we write for all $s \in [\tau, \tau_1]$,
\begin{align*}
\|q_e(s)\|_{L^\infty} &\leq e^{-\frac{s - \tau}{p-1}}\|q_e(\tau)\|_{L^\infty} + \int_{\tau}^s e^{-\frac{s - s'}{p-1}}\left\|(1 - \chi)(Q(s') + R(s')) + E_1(s') \right\| _{L^\infty}ds'\\
& \quad + C\sum_{j = 1}^{2m-1}\int_{\tau}^s e^{-\frac{s - s'}{p-1}} \big(1 - e^{(s - s')}\big)^{-\frac{j}{2m}}\left\|q(s') \nabla^{2m - j}\chi(s')\right\|_{L^\infty}ds'.
\end{align*}
For $K$ large enough, we have 
$$\| (1 - \chi(y,s')Q(s'))\|_{L^\infty} \leq C\|\varphi(s')\|_{L^\infty}^{p-1}\|q_e(s')\|_{L^\infty} \leq \frac{1}{2(p-1)}\|q_e(s')\|_{L^\infty}.$$
Recall from Lemma \ref{lemm:exR} that $\|R(s')\|_{L^\infty} \leq \frac{C}{s'}$. As for $E_1$, we use the definition of $\chi$ given in \eqref{def:chi} and the  bounds given in Definition \ref{def:VA} to obtain the estimate 
$$\|E_1(s')\|_{L^\infty} \leq C\|q(s')\|_{L^\infty\big(Ks'^{\frac 1{2m}}\leq |y| \leq 2Ks'^{\frac{1}{2m}} \big)} \leq \frac{CA^{M+1}}{s'^{\frac{1}{2m}}}.$$
Since $\|\nabla^{2m - j}\chi( s')\|_{L^\infty} \leq Cs'^{-\frac{2m - j}{2m}}$, using the fact that $\nabla^{2m - j}\chi$ is compactly supported, we estimate
$$\sum_{j = 1}^{2m-1}\left\|q(s') \nabla^{2m - j}\chi(s')\right\|_{L^\infty} \leq \frac{CA^{M+1}}{s'^{\frac{1}{m}}}.$$
A collection of these above estimates yields 
\begin{align*}
\|q_e(s)\|_{L^\infty} &\leq e^{-\frac{s - \tau}{p-1}}\|q_e(\tau)\|_{L^\infty}\\
& + \int_{\tau}^s e^{-\frac{s - s'}{p-1}} \left( \frac{1}{2(p-1)}\|q_e(s')\|_{L^\infty} + \frac{CA^{M+1}}{s'^{\frac{1}{2m}}} +   \frac{CA^{M+1}}{\Big[s'\sqrt{1 - e^{(s - s')}} \, \Big]^{\frac 1m}}\right)ds'.
\end{align*}
Applying the standard Gronwall's inequality to the function $e^\frac{s}{p-1}\|q_e(s)\|_{L^\infty}$ yields
$$e^\frac{s}{p-1}\|q_e(s)\|_{L^\infty} \leq e^{\frac{s - \tau}{2(p-1)}}e^{\frac{\tau}{p-1}}\lambda(\tau) + \frac{CA^{M+1}}{s^\frac{1}{2m}}(s-\tau + 1),$$
from which the estimate \eqref{eq:contr:qe} follows. This completes the proof of Proposition \ref{prop:dya}. \hfill $\square$


\def\cprime{$'$}

\end{document}